\documentclass[12pt,a4paper]{amsart}

\usepackage{graphics,epic}
\usepackage{amsmath,amssymb, amsthm}
\usepackage[all,2cell]{xy}
\usepackage{mathrsfs}
\usepackage[usenames]{color}
\newtheorem{theorem}{Theorem}[section]
\newtheorem*{theorem*}{Theorem}

\newtheorem{lemma}[theorem]{Lemma}
\newtheorem{proposition}[theorem]{Proposition}
\newtheorem{corollary}[theorem]{Corollary}
\newtheorem{conjecture}[theorem]{Conjecture}
\newtheorem*{conjecture*}{Conjecture}

\newtheorem*{question*}{Question}
\newtheorem{example}[theorem]{Example}
\newtheorem{remark}[theorem]{Remark}

\newcommand{\ie}{{\em i.e.~}\ }

\newcommand{\eg}{{\em e.g.~}\ }

\newcommand{\opname}[1]{\operatorname{\mathsf{#1}}}

\renewcommand{\mod}{\opname{mod}\nolimits}
\newcommand{\grmod}{\opname{grmod}\nolimits}
\newcommand{\proj}{\opname{proj}\nolimits}

\newcommand{\Inj}{\opname{Inj}\nolimits}
\newcommand{\Mod}{\opname{Mod}\nolimits}

\newcommand{\Grmod}{\opname{Grmod}\nolimits}

\newcommand{\grproj}{\opname{grproj}}
\newcommand{\Grproj}{\opname{Grproj}}

\newcommand{\rep}{\opname{rep}\nolimits}
\newcommand{\Rep}{\opname{Rep}\nolimits}

\newcommand{\im}{\opname{im}\nolimits}

\newcommand{\thick}{\opname{thick}\nolimits}
\newcommand{\Tria}{\opname{Tria}\nolimits}
\newcommand{\per}{\opname{per}\nolimits}

\newcommand{\Hom}{\opname{Hom}}
\newcommand{\End}{\opname{End}}

\newcommand{\cHom}{\mathcal{H}\it{om}}
\newcommand{\cEnd}{\mathcal{E}\mathit{nd}}

\newcommand{\rad}{\opname{rad}}

\newcommand{\ten}{\otimes}
\newcommand{\lten}{\overset{\boldmath{L}}{\ten}}

\newcommand{\Cone}{\opname{Cone}}

\newcommand{\Tot}{\opname{Tot}}

\newcommand{\ca}{{\mathcal A}}
\newcommand{\cb}{{\mathcal B}}
\newcommand{\cc}{{\mathcal C}}
\newcommand{\cd}{{\mathcal D}}

\newcommand{\ch}{{\mathcal H}}
\newcommand{\ci}{{\mathcal I}}

\newcommand{\cp}{{\mathcal P}}

\newcommand{\ct}{{\mathcal T}}

\newcommand{\cx}{{\mathcal X}}
\newcommand{\cy}{{\mathcal Y}}
\newcommand{\cz}{{\mathcal Z}}

\newcommand{\dslash}{/\hspace{-5pt}/}

\begin{document}

\title[Gentle algebras]{Derived categories of graded gentle one-cycle algebras}

\author{Martin Kalck}
\thanks{M.K. was supported by DFG grant Bu--1866/2--1.}
\address{Martin Kalck, The Maxwell Institute, School of Mathematics, James Clerk Maxwell Building, The King's Buildings, Mayfield Road, Edinburgh, EH9 3JZ, UK.}
\email{m.kalck@ed.ac.uk}
\author{Dong Yang}
\address{Dong Yang, Department of Mathematics, Nanjing University, Nanjing 210093, P. R. China}
\email{yangdong@nju.edu.cn}
\thanks{D.Y. was supported by Max-Planck-Institute for Mathematics in Bonn, the DFG program SPP 1388 (YA297/1-1 and KO1281/9-1), the National Science Foundation in China No. 11401297 and a project funded by the Priority Academic Program Development of Jiangsu Higher Education Institutions.}

\date{\today}
\maketitle

\bigskip

\begin{abstract}
Let $A$ be a graded algebra. It is shown that the derived category of dg modules over $A$ (viewed as a dg algebra with trivial differential) is a triangulated hull of a certain orbit category of the derived category of graded $A$-modules. This is applied to study derived categories of graded gentle one-cycle algebras.\\
{\bf Keywords}: derived category, derived equivalence, graded algebra, orbit category, gentle one-cycle algebra.\\
{\bf
2010 MSC}: 16E35, 16E45, 18E30.
\end{abstract}


\section{Introduction}

Discrete triangulated categories are, roughly speaking, those Krull--Schmidt triangulated categories
which do not admit `continuous' families of isomorphism classes of
indecomposable objects (see \cite{Vossieck01,BroomheadPauksztelloPloog15} for various notions of discreteness). A special class of such categories called locally finite triangulated categories (\eg those with finitely many
isomorphism classes of indecomposable objects) were
intensively studied, in particular, their Auslander--Reiten quivers are classified, see~\cite{XiaoZhu05,Amiot07a}.
For a finite-dimensional algebra (over an algebraically closed field), Vossieck's theorem~\cite{Vossieck01} states that its derived
category is discrete if and only if it is derived equivalent to a hereditary algebra of finite representation type (namely, the path algebra
of a Dynkin quiver) or it is a gentle one-cycle algebra which does not satisfy the \emph{clock condition} (see Section~\ref{s:gentle-one-cycle-alg}). The Auslander--Reiten quiver was determined in the former case by Happel in~\cite{Happel87} and in the latter case by Bobi\'nski--Geiss--Skowro\'nski in~\cite{BobinskiGeissSkowronski04}. See \cite{BobinskiKrause15,ArnesenLakingPauksztelloPrest16,BroomheadPauksztelloPloog13,BroomheadPauksztelloPloog16,Qin16} for further study on discrete derived categories.

Recently, certain discrete triangulated categories of geometrical origin have been studied, \eg the triangulated category generated by a $d$-spherical object ~\cite{KellerYangZhou09} and the relative singularity category of the Auslander resolution of the nodal curve singularity \cite{BurbanKalck12}. They turn out to be derived categories of dg modules over certain graded gentle one-cycle algebras, more precisely, $k[x]/x^2$ with $\deg(x)=d$ and the path algebra of the graded quiver $\xymatrix{1\ar@<.4ex>[r]^{\alpha}&2\ar@<.4ex>[l]^{\beta}}$
with both arrows in degree $-1$, respectively.
Moreover, derived categories of graded hereditary algebras of type $\tilde{A_n}$ are triangle equivalent to partially wrapped Fukaya categories of graded annuli, see \cite[Sections 1.2 and 6.3]{HaidenKatzarkovKontsevich14} and \cite[Section 2.1]{LekiliPolishchuk17}. Our Theorem 1.2 below gives a representation theoretic description of these triangulated Fukaya categories and gives a partial answer to the following question.

\begin{question*} When is the derived category of dg modules over a graded gentle one-cycle algebra (viewed as a dg algebra with trivial differential)
discrete and what does the
Gabriel/Auslander--Reiten quiver look like?
\end{question*}

In this paper we are not able to define derived discreteness for graded algebras, but with Theorem~\ref{t:main-thm-1.2} we believe that the graded algebras $\Gamma(p,q,r)$ ($r\in\mathbb{Z}\backslash\{0\}$) and $\Gamma'(q,r)$ ($r\in\mathbb{Z}$) in Theorem~\ref{t:main-thm-1.1} are derived discrete for a reasonable definition of derived discreteness.

\begin{theorem} \label{t:main-thm-1.1}
Let $A$ be a graded gentle one-cycle algebra.
 \begin{itemize}
  \item[(a)] If $A$ has finite global dimension, then there is a triple $(p,q,r)$ of integers
with $p,q\in\mathbb{N}$\hspace{1pt}\footnote{$\mathbb{N}$ is the set of positive integers.} and $r\in\mathbb{Z}$ such that $A$ is derived equivalent to the path algebra
$\Gamma(p,q,r)$ of the graded quiver
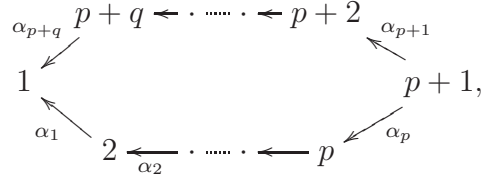
\[\xymatrix@R=0.7pc@C=0.7pc{&p+q\ar[dl]_{\alpha_{p+q}}&\cdot\ar[l]\ar@{.}[r]&\cdot&p+2\ar[l]&\\
1&&&&&p+1\ar[dl]^{\alpha_p}\ar[ul]_{\alpha_{p+1}},\\
&2\ar[ul]^{\alpha_1}&\cdot\ar[l]^{\alpha_2}\ar@{.}[r]&\cdot&p\ar[l]&}\]
where $\deg(\alpha_i)=\delta_{i,p+q}r$.

\item[(b)] If $A$ has infinite global dimension, then there are integers $q\in\mathbb{N}$ and $r\in\mathbb{Z}$ such that $A$ is derived equivalent to the quotient algebra $\Gamma'(q,r)$ of
the  path algebra of the graded quiver
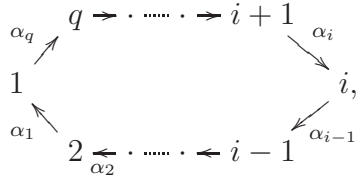
\[\xymatrix@R=0.7pc@C=0.7pc{&q\ar[r]&\cdot\ar@{.}[r]&\cdot\ar[r]&i+1\ar[dr]^{\alpha_i}&\\
1\ar[ur]^{\alpha_{q}}&&&&&i\ar[dl]^{\alpha_{i-1}},\\
&2\ar[ul]^{\alpha_1}&\cdot\ar[l]^{\alpha_2}\ar@{.}[r]&\cdot&i-1\ar[l]&}\]
modulo all paths of length two, where $\deg(\alpha_i)=\delta_{i,q}(q-r)$.
 \end{itemize}
\end{theorem}

For a dg algebra $A$, let $\cd_{fd}(A)$ denote the full subcategory of the derived category of $A$ consisting of those dg $A$-modules with finite-dimensional total cohomology.

\begin{theorem} \label{t:main-thm-1.2}
Let $p,q\in\mathbb{N}$ and $r\in\mathbb{Z}\backslash\{0\}$.
 \begin{itemize}

\item[(a)]  The Auslander--Reiten quiver of $\cd_{fd}(\Gamma(p,q,r))$ has $3|r|$ connected components:
$\cx^1_i$ of type $\mathbb{Z}A_\infty$, $\cx^2_i$ of type $\mathbb{Z}A_\infty$ and $\cp_i$ of type $\mathbb{Z}A_\infty^\infty$, where
$0\leq i\leq |r|-1$. The suspension functor defines cyclic permutations of order $|r|$ on the sets
$\{\cx^1_i\}$, $\{\cx^2_i\}$ and $\{\cp_i\}$, respectively. For $X\in\cx^1_i$ we have $\tau^pX=\Sigma^{r}X$and for $X\in\cx^2_i$ we have $\tau^qX=\Sigma^{-r}X$. In other words, objects in $\cx^1_i$ and $\cx^2_i$ are fractionally Calabi--Yau of dimension $\frac{p+r}{p}$ and $\frac{q-r}{q}$, respectively.
\item[(b)]  The Gabriel quiver of $\cd_{fd}(\Gamma'(q,r))$ has $2|r|$ connected components: $\cx_i$ of type $\mathbb{Z}A_\infty$
 and $\cy_i$ of type linear $A_\infty^\infty$, where $0\leq i\leq |r|-1$. The suspension functor defines cyclic permutations of order $|r|$ on the sets
$\{\cx_i\}$ and $\{\cy_i\}$, respectively. The triangulated subcategory generated by objects in $\cx_i$ has Auslander--Reiten triangles. For $X\in\cx_i$  we have $\tau^{q}X=\Sigma^{-q+r}X$. In other words, objects in $\cx_i$ are fractionally Calabi--Yau of dimension $\frac{r}{q}$.
\item[(c)] The Gabriel quiver of $\cd_{fd}(\Gamma'(q,0))$ consists of $\mathbb{Z}$ tubes $\cx_i$ of rank $q$ and $\mathbb{Z}$ cyclic quivers $\cp_i$ with $q$ vertices. The triangulated subcategory generated by objects in $\cx_i$ has Auslander--Reiten triangles. For $X\in\cx_i$, we have $\tau^qX=X$. In other words, objects in $\cx_i$ are fractionally Calabi--Yau of dimension $\frac{q}{q}$.
 \end{itemize}
\end{theorem}

From Theorem~\ref{t:main-thm-1.2} (a) we deduce by a straightforward calculation that $\Gamma(p,q,r)$ and $\Gamma(p',q',r')$ are derived equivalent if and only if $(p,q,r)=(p',q',r')$ or $(p,q,r)=(q',p',-r')$. If $(p,q,r)=(q',p',-r')$, the two graded algebras are actually graded equivalent. Similarly, $\Gamma'(q,r)$ and $\Gamma'(q',r')$ are derived equivalent if and only if $(q,r)=(q',r')$.

\medskip

Part of the strategy to prove Theorem~\ref{t:main-thm-1.1} is to find
a graded version of the old results and push them forward to the dg
version. To this end, we investigate the relationship between
various derived categories associated to \emph{any} graded algebra
$A$. These results are interesting in their own right. More
precisely, we study the following diagrams of triangle functors,
which forms the main ingredients of the homological Koszul
duality~\cite{RicheSoergelWilliamson14}:
\[\xymatrix@C=0.5pc{&\cd(\Grmod A)\ar[ld]_{\Tot}\ar[rd]^{F}& &&\ch^b(\grproj A)\ar[ld]_{\Tot}\ar[rd]^{F}&\\
\cd(A) & & \cd(\Mod FA) & \per(A) & & \ch^b(\proj FA)}
\]
where $F$ is the functor of
forgetting the grading, $\cd(A)$ is the derived category of dg
$A$-modules, $\per(A)$ is the triangulated subcategory of $\cd(A)$
generated by the free dg module of rank $1$. A complex of graded $A$-modules can be viewed as a bicomplex. The functor $\Tot$  takes such a complex to its total complex. The right hand sides, which deals
with the forgetful functor, has been intensively studied in the
literature, see for example~\cite{Marcus03}. So our attention is focused on the left hand side.

\begin{theorem}\label{t:main-thm-2}
\begin{itemize}
\item[(a)] Let $A$ be a graded algebra.
The functor $\Tot:\ch^b(\grproj A)\rightarrow\per(A)$
induces a fully faithful functor $\overline{\Tot}:\ch^b(\grproj A)/\Sigma\circ\langle-1\rangle\rightarrow \per(A)$, where $\ch^b(\grproj A)/\Sigma\circ\langle-1\rangle$ is the orbit category of $\ch^b(\grproj A)$ with respect to the auto-equivalence $\Sigma\circ\langle-1\rangle$,
so that $\per(A)$ becomes a triangulated hull of $\ch^b(\grproj A)/\Sigma\circ\langle-1\rangle$.
When $A$ is graded hereditary,
$\overline{\Tot}$ is an equivalence.
\item[(b)] Let $A$ and $B$ be two graded algebras, and $T$ be a complex of graded $B$-$A$-bimodules. The following assertions are equivalent
\begin{itemize}
 \item[(i)] $?\lten_B T:\cd(\Grmod B)\rightarrow\cd(\Grmod A)$ is a triangle equivalence,
 \item[(ii)] $?\lten_B \Tot(T):\cd(B)\rightarrow\cd(A)$ is a triangle equivalence,
 \item[(iii)] $?\lten_{FB} F(T):\cd(\Mod FB)\rightarrow\cd(\Mod FA)$ is a triangle equivalence.
\end{itemize}

\end{itemize}
\end{theorem}

The equivalence between (i) and (iii) in (b) is part of~\cite[Theorem 2.4]{Marcus03}.

The paper is organized as follows. In Section~\ref{s:derived-cat} we
recall results on derived categories and their dg enhancements.
Section~\ref{s:orbit-category} is devoted to the study of orbit categories and Section~\ref{s:graded-alg-as-dg-alg} to
proving Theorem~\ref{t:main-thm-2}. In
Section~\ref{s:bigraded-dg-alg} we give a sufficient condition for the formality of a bigraded dg algebra and its totalization, which will be used in Section~\ref{s:derived-equi-graded-alg}. In Section~\ref{s:derived-equi-graded-alg} we study a condition under which
derived equivalent graded algebras are still derived equivalent
after a compatible change of gradings. In
Section~\ref{s:gentle-one-cycle-alg} we recall the definition of
gentle one-cycle algebras and recall results by
Assem--Skowro\'{n}ski, Vossieck and
Bobi\'nski--Geiss--Skowro\'nski.
 In Section~\ref{s:graded-a-pq} the derived category of the path algebra of a graded quiver of type $\tilde{A}$ is investigated and Theorem~\ref{t:main-thm-1.2} is proved in Sections~\ref{sss:AR-quiver-no-cycle} and~\ref{sss:AR-quiver-truncated-cycle}.
In Section~\ref{s:graded-gentle-one-cycle-alg} we prove Theorem~\ref{t:main-thm-1.1}.

\smallskip

\noindent{\it Acknowledgements.}
The second-named author thanks Zhe Han for some helpful
conversations. He expresses his deep gratitude to Bernhard Keller
for answering various questions on graded algebras, pre-triangulated
dg categories and orbit categories. The authors thank the referee for very helpful comments, including pointing out an error in the description of the Gabriel quiver in Section~\ref{sss:zigzag} and pointing out that in Theorem~\ref{thm:graded-gentle-one-cycle-is-affine-A} the assumption `$k$ is algebraically closed' is not necessary.

\section{Derived categories}\label{s:derived-cat}

By abuse of notation, we will denote by $\Sigma$ the suspension
functor of all triangulated categories. Throughout, $k$ will be a
field and $\ten=\ten_k$. All categories in this paper are $k$-categories and all functors are $k$-linear.

\subsection{Generators of triangulated
categories}\label{ss:generators}

Let $\cc$ be a triangulated category. For a set $\{X_i| i\in I\}$ of objects of
$\cc$, we denote 
by $\thick_{\cc}(X_i| i\in I)$ the
smallest triangulated subcategory of $\cc$ containing all $X_i$ and closed
under taking direct summands, and by $\Tria_{\cc}(X_i| i\in I)$ the smallest
triangulated subcategory of $\cc$ containing all $X_i$ and closed under
taking all existing direct sums in $\cc$. If it does not cause
confusion, we will omit the subscripts. If $\cc=\thick(X_i| i\in I)$, we call
$\{X_i| i\in I\}$ a set of \emph{generators} of $\cc$ and we say that $\{X_i| i\in I\}$ \emph{generates}
$\cc$.

Assume that $\cc$ has infinite (set-indexed) direct sums. An object
$X$ of $\cc$ is \emph{compact} if $\Hom_{\cc}(X,?)$ commutes with
all direct sums. A set of objects $\{X_i|i\in I\}$ is a set of \emph{compact generators} of $\cc$ if all $X_i$ are compact and
$\cc=\Tria(X_i|i\in I)$.

\begin{theorem}[{\cite[Theorem 5.3]{Keller94}}]\label{thm:compact-objects}
Assume that $\cc$ has a set of compact generators $\{X_i|i\in I\}$. Then $X\in\cc$ is compact if and only if it belongs to $\thick_\cc(X_i|i\in I)$.
\end{theorem}

It is well-known that if $\cc$ has a set of compact generators $\{X_i|i\in I\}$ then a set $\{Y_j|j\in J\}$ of compact objects of $\cc$ is a set of compact generators if and only if $\thick_\cc(Y_j|j\in J)=\thick_\cc(X_i|i\in I)$ if and only if $\Hom_\cc(Y_j,\Sigma^p Z)=0$ for all $j\in J$ and all $p\in\mathbb{Z}$ implies that $Z\cong 0$. A proof is given using \cite[Theorem 2.1]{Neeman92a}, \cite[Theorem 5.2]{Keller94} and  \cite[Theorem 9.1.16 and Proposition 9.1.19]{Neeman01b}.

\subsection{DG categories and dg
enhancements}\label{ss:dg-enhancements}

 We follow~\cite{BondalKapranov90}.

Let $\ca$ be a \emph{dg category}, \ie a category in which all
the $k$-vector spaces $\Hom_{\ca}(X,Y)$ are endowed with the
structure of a (cochain) complex such that the compositions
\[\Hom_{\ca}(Y,Z)\ten \Hom_{\ca}(X,Y)\rightarrow \Hom_{\ca}(X,Z)\]
are chain maps of complexes. We define $H^0\ca$ (respectively, $Z^0\ca$) to be the
$k$-category whose objects are the same as $\ca$ and whose morphism
space $\Hom_{H^0\ca}(X,Y)$ is the 0-th cohomology of the complex
$\Hom_{\ca}(X,Y)$ (respectively, the set of closed morphisms from $X$ to $Y$ of degree $0$). A \emph{dg functor} $F:\ca\rightarrow\cb$ between
dg categories is a  functor such that each morphism
$F(X,Y):\Hom_{\ca}(X,Y)\rightarrow\Hom_{\cb}(FX,FY)$ is a chain map
of complexes. We define $H^0F$ to be the $k$-linear functor
$H^0F:H^0\ca\rightarrow H^0\cb$ such that $H^0FX=FX$ and
$(H^0F)(X,Y)=H^0F(X,Y)$.

The category $\cc_{dg}(k)$ whose objects are complexes of
$k$-vector spaces and whose morphism space
$\Hom_{\cc_{dg}(k)}(X,Y)$ is $\cHom_k(X,Y)$ given by
\[\cHom_k^n(X,Y)=\prod_{m\in\mathbb{Z}}\Hom_k(X^m,Y^{m+n})\]
with $d(f)=d_Y\circ f-(-1)^n f\circ d_X$ for $f\in\cHom_k^n(X,Y)$ is a
dg category. We denote by $[1]:\cc_{dg}(k)\rightarrow\cc_{dg}(k)$
the shift functor, which is a dg functor mapping a homogeneous
morphism $f$ of degree $n$ to $(-1)^n f$.

Let $\ca$ be a dg category. For any object $X$ of $\ca$, the
functors $\Hom_{\ca}(X,?):\ca\rightarrow\cc_{dg}(k)$ and
$\Hom_{\ca}(?,X):\ca^{op}\rightarrow\cc_{dg}(k)$ are dg functors.
Let $X$ be an object of $\ca$, we define a dg functor
$F_X:Y\mapsto \Hom_{\ca}(Y,X)[1]$. For
$f:X\rightarrow Y$ a closed morphism of degree $0$, we define a dg
functor $F_f:Z\mapsto \Cone(\Hom_{\ca}(Z,f))$. The dg category $\ca$
is said to be \emph{pre-triangulated} if it has a zero object and the dg functors $F_X$
($X\in\ca$) and $F_f$ ($f$ closed morphism of degree $0$) are representable, that is, there are objects $X[1]$ and  $\Cone(f)$ of $\ca$ with isomorphisms of dg functors $F_X\cong \Hom_{\ca}(?,X[1])$ and $F_f\cong \Hom_{\ca}(?,\Cone(f))$.

\begin{theorem}\emph{(\cite[Proposition 3.2]{BondalKapranov90})}
Let $\ca$ be a pre-triangulated category. Then $H^0\ca$ carries a
natural triangle structure with the suspension functor being
$\Sigma=H^0[1]$. The triangle associated to a morphism $f:X\rightarrow Y$ in $H^0\ca$ is \[X\stackrel{f}{\rightarrow} Y\rightarrow \Cone{\tilde{f}}\rightarrow X[1],\]
where $\tilde{f}$ is a lift of $f$ in $\ca$ and the last two morphisms are canonical morphisms associated to $\Cone(\tilde{f})$.
\end{theorem}

Let $\cc$ be a triangulated category and $\ca$ be a dg category. We
say that $\ca$ is a \emph{dg enhancement} of $\cc$ if $\ca$ is
pre-triangulated and $\cc$ is triangle equivalent to $H^0\ca$
endowed with the above natural triangle structure. For $\cc'$ a triangulated subcategory of $\cc$, let $\ca'$ be the dg subcategory of $\ca$ consisting of objects in the essential image of $\cc'$ in $H^0\ca$. Then $\ca'$ is a pre-triangulated dg subcategory of $\ca$ and is a dg enhancement of $\cc'$.

Let $\ca$ and $\cb$ be two pre-triangulated dg categories and $F:\ca\rightarrow\cb$ be a dg functor. Then $F$ canonically commutes with $[1]$ and $\Cone$, and induces a triangle functor $H^0F:H^0\ca\rightarrow H^0\cb$.

\subsection{Derived categories of abelian
categories}\label{ss:derived-cat-of-abelian-cat}

Let $\ca$ be an additive category. For $*\in\{b,+,-,\emptyset\}$,
let $\ch^*(\ca)$ (respectively, $\cc(\ca)$) be the homotopy category (respectively, the category) of complexes of objects in
$\ca$ satisfying the corresponding boundedness condition.

For $X$ and $Y$ two complexes of objects in $\ca$, we define
$\cHom_{\ca}(X,Y)$ to be the complex whose degree $n$ component is
\[\cHom_{\ca}^n(X,Y)=\prod_{m\in\mathbb{Z}}\Hom_{\ca}(X^m,Y^{m+n})\]
and whose differential is given by
\[d(f)=d_Y\circ f-(-1)^n f\circ d_X\]
for $f\in\cHom_\ca^n(X,Y)$. For $*\in\{b,+,-,\emptyset\}$, we define
$\cc_{dg}^*(\ca)$ to be the category whose objects are complexes of
objects in $\ca$ satisfying the corresponding boundedness condition
and whose morphism spaces are
$\Hom_{\cc_{dg}^*(\ca)}(X,Y)=\cHom_{\ca}(X,Y)$. The dg shift functor $[1]:\cc_{dg}^*(\ca)\rightarrow\cc_{dg}^{*}(\ca)$ takes value the complex shift $X[1]$ on an object $X$ and sends a homogeneous morphism $f$ of degree $n$ to $(-1)^n f$. For a closed morphism $f:X\rightarrow Y$ of degree $0$, the dg functor $F_f$ is represented by the mapping cone of $f$.

\begin{proposition}[\cite{BondalKapranov90}]
The category $\cc_{dg}^*(\ca)$ is a pre-triangulated dg category
such that $Z^0\cc_{dg}(\ca)=\cc(\ca)$ and $H^0\cc_{dg}^*(\ca)=\ch^*(\ca)$.
\end{proposition}

Let $\ca$ be an abelian category. For a $*\in\{b,+,-,\emptyset\}$,
the \emph{derived category} $\cd^*(\ca)$ is the triangle quotient of
the homotopy category $\ch^*(\ca)$ by its subcategory of acyclic
complexes of objects in $\ca$. Following~\cite{Spaltenstein88}, we
say that an object $X$ of $\cc_{dg}(\ca)$ is \emph{$\ch$-projective} if
$\cHom_{\ca}(X,?)$ preserves acyclicity. Complexes of projective objects bounded from the right are $\ch$-projective.

Assume further that $\ca$ has enough projectives, direct limits exist in $\ca$ and taking direct limit is exact. Then

\begin{proposition}[{\cite[Theorem 3.4 and Corollary 3.5]{Spaltenstein88}}] Every object $M$ of $\cc_{dg}(\ca)$ has an $\ch$-projective resolution, \ie a quasi-isomorphism $\mathrm{p}M\rightarrow M$ with $\mathrm{p}M$ being $\ch$-projective. Moreover, each $\ch$-projective object is isomorphic in $\ch(\ca)$ to a \emph{cofibrant object}, \ie an object $F$ which admits a filtration
\[
0=F_0\subseteq F_1\subseteq\ldots\subseteq F_p\subseteq F_{p+1}\subseteq\ldots\subseteq F
\]
such that
\begin{itemize}
\item[(F1)] $F$ is the union of the $F_p$, $p\in\mathbb{N}\cup\{0\}$,
\item[(F2)] the inclusion morphism $F_{p-1}\subseteq F_p$ splits in every degree, for any $p\in\mathbb{N}$,
\item[(F3)] the subquotient $F_p/F_{p-1}$ is isomorphic in $Z^0\cc_{dg}(\ca)$ to a direct summand of a direct sum of shifts of projective objects in $\ca$, for any $p\in\mathbb{N}$.
\end{itemize}
\end{proposition}

As a consequence, the quotient functor $\ch(\ca)\to\cd(\ca)$ admits a left adjoint $\cd(\ca)\to\ch(\ca)$, which is fully faithful and  identifies $\cd(\ca)$ with the full subcategory of $\ch(\ca)$ consisting of $\ch$-projective objects.

\subsection{Derived categories of dg
algebras}\label{ss:derived-cat-of-dg-alg}

This subsection is parallel to the preceding subsection. We
follow~\cite{Keller94,Keller06d,Keller98c}.

Let $A$ be a \emph{dg algebra}, \ie a dg category with one object.
More precisely, $A=\bigoplus_{i\in\mathbb{Z}}A^i$ is a graded $k$-algebra endowed with a differential $d_A$ of degree $1$ such that the graded Leibniz rule holds
\[d_A(aa')=d_A(a)a'+(-1)^{|a|}ad_A(a'),\]
where $a$ is homogeneous of degree $|a|$.
A (right) \emph{dg
$A$-module} is a complex $M$ (of $k$-vector spaces) endowed with an
$A$-action from the right
\[M^i\ten A^j\rightarrow M^{i+j}, m\ten a\mapsto ma\]
such that the graded Leibniz rule holds
\[d_M(ma)=d_M(m)a+(-1)^{|m|}md_A(a),\]
where $m\in M$ is homogeneous of degree $|m|$.

For two dg $A$-modules $M$ and $N$, we define $\cHom_A(M,N)$ to be
the subcomplex of $\cHom_k(M,N)$ whose degree $i$ component is the
subspace
\begin{eqnarray*}\cHom_A^i(M,N)&=&\{f\in\prod_{j\in\mathbb{Z}}\Hom_k(M^j,N^{i+j})|\\
&&f(ma)=f(m)a\text{ for any }  m\in M \text{ and any } a\in
A\}\\
&=&\Hom_{\Grmod A}(M,N\langle i\rangle).\end{eqnarray*} Let
$\cc_{dg}(A)$ denote the category whose objects are dg $A$-modules
and whose morphism spaces are
$\Hom_{\cc_{dg}(A)}(M,N)=\cHom_A(M,N)$. For a dg $A$-module $M$, one checks that the complex shift $X[1]$ is still a dg $A$-module with the $A$-action unchanged. Actually $[1]:\cc_{dg}(A)\rightarrow\cc_{dg}(A)$ a dg
functor, it sends a homogeneous morphism $f$ of degree $n$ to
$(-1)^n f$. For a closed morphism $f:M\rightarrow N$ of degree $0$, it is easy to check that the mapping cone of $f$ is again a dg $A$-module and it represents the dg functor $F_f$.

\begin{proposition}
The category $\cc_{dg}(A)$ is a pre-triangulated dg category.
\end{proposition}

The triangulated category $\ch(A)=H^0\cc_{dg}(A)$ is called the
\emph{homotopy category} of dg $A$-modules. Then the
\emph{derived category} of dg $A$-modules (or the \emph{derived
category} of $A$), denoted by $\cd(A)$, is defined as the triangle quotient of $\ch(A)$ by the
(triangulated) subcategory of acyclic complexes. We will denote
$\per(A)=\thick_{\cd(A)}(A)$, and denote by $\cd_{fd}(A)$ the triangulated
subcategory of $\cd(A)$ consisting of those dg $A$-modules whose
total cohomology is finite-dimensional. When $A$ is a finite-dimensional
algebra viewed as a dg algebra concentrated in degree $0$, we have
triangle equivalences $\cd(A)=\cd(\Mod A)$, $\cd^b(\mod A)\simeq\cd_{fd}(A)$ and $\ch^b(\proj A)\simeq\per(A)$.

 A dg
$A$-module $M$ is said to be \emph{$\ch$-projective} if the dg
functor $\cHom_A(M,?):\cc_{dg}(A)\rightarrow \cc_{dg}(k)$ preserves
acyclicity.

\begin{theorem}[{\cite[Theorem 3.1]{Keller94}}]\label{t:existence-of-resolutions}
Any dg $A$-module admits an $\ch$-projective resolution, \ie a quasi-isomorphism $\mathrm{p}M\rightarrow M$ with $\mathrm{p}M$ being $\ch$-projective. Moreover, each $\ch$-projective dg module is isomorphic in $\ch(\ca)$ to a \emph{cofibrant dg module}, \ie a dg module $F$ which admits a filtration
\[
0=F_0\subseteq F_1\subseteq\ldots\subseteq F_p\subseteq F_{p+1}\subseteq\ldots\subseteq F
\]
such that
\begin{itemize}
\item[(F1)] $F$ is the union of the $F_p$, $p\in\mathbb{N}$,
\item[(F2)] the inclusion morphism $F_{p-1}\subseteq F_p$ splits in every degree, for any $p\in\mathbb{N}$,
\item[(F3)] the subquotient $F_p/F_{p-1}$ is isomorphic in $Z^0\cc_{dg}(A)$ to a direct summand of a direct sum of shifts of $A$, for any $p\in\mathbb{N}$.
\end{itemize}
\end{theorem}

As a consequence, the quotient functor $\ch(A)\to\cd(A)$ admits a left adjoint $\cd(A)\to\ch(A)$, which is fully faithful and identifies $\cd(A)$ with the full subcategory of $\ch(A)$ consisting of $\ch$-projective dg $A$-modules.

\begin{theorem}[\cite{Keller94}]\label{t:compact-generator}
\begin{itemize}
\item[(a)] The free dg $A$-module of rank $1$ is a compact generator of the triangulated category $\cd(A)$.
\item[(b)] A dg $A$-module is compact in  $\cd(A)$ if and only it belongs to
$\per(A)$, and it is a compact generator if and only if it is a
generator of $\per(A)$.
\end{itemize}
\end{theorem}

\begin{lemma}\label{lem:der-equiv-induces-equiv-on-per-and-dfd}
A derived equivalence $\cd(B)\rightarrow\cd(A)$ of dg algebras
restricts to triangle equivalences $\per(B)\rightarrow\per(A)$ and
$\cd_{fd}(B)\rightarrow\cd_{fd}(A)$.
\end{lemma}

The following well-known result is a consequence of~\cite[Lemma 6.1
(a)]{Keller94}.

\begin{proposition}\label{p:derived-equiv}
\begin{itemize}
\item[(a)] Let $M$ be a compact generator of $\cd(A)$ and assume that $M$
is $\ch$-projective. Let $B=\cEnd_A(M)=\cHom_A(M,M)$ be the dg
endomorphism algebra of $M$. Then the derived functor $?\lten_B
M:\cd(B)\rightarrow\cd(A)$ is a triangle equivalence.
\item[(b)] For a
quasi-isomorphism of dg algebras $A\rightarrow C$, the induction
functor $?\lten_A C:\cd(A)\rightarrow\cd(C)$ is a triangle equivalence.
\end{itemize}
\end{proposition}

\begin{proposition} Let $T\in\cd(A)$ be a tilting object, \ie a compact generator such that $\Hom_{\cd(A)}(T,\Sigma^n T)=0$ for $n\neq 0$. Then $\cd(A)$ and $\cd(\End_{\cd(A)}(T))$ are triangle equivalent.
\end{proposition}

\subsection{Formal dg algebras}

Let $A$ be a dg algebra. Then its total cohomology
$H^*A=\bigoplus_{n\in\mathbb{Z}}H^nA$ admits the induced structure
of a graded algebra. We view it as a dg algebra with trivial
differential. The dg algebra $A$ is said to be \emph{formal} if it is
related by a zigzag of quasi-isomorphisms to $H^*A$.

\begin{lemma}\label{l:formal-dg-alg} Let $A$ be a formal dg algebra. Then $\cd(A)$ and $\cd(H^*A)$ are triangle equivalent.
\end{lemma}
\begin{proof} This is a consequence of
Proposition~\ref{p:derived-equiv} (b).
\end{proof}

\section{Orbit categories}\label{s:orbit-category}

In this section we study orbit categories. We define large orbit categories for categories with infinite (set-indexed) direct sums with an auto-equivalence, establish the universal property and compare compact objects of a category and its orbit category.

\subsection{The orbit category}\label{ss:orbit-category}
Let $\ca$ be a category and $\Phi:\ca\rightarrow\ca$ be an
auto-equivalence. We define the \emph{orbit category} $\ca/\Phi$
(\cite{Keller05}): it has the same objects as $\ca$, and the
morphism space from $X$ to $Y$ is defined by
\[\Hom_{\ca/\Phi}(X,Y)
=\bigoplus_{p\in\mathbb{Z}}\Hom_{\ca}(X,\Phi^p Y).\] Let $\pi:\ca\to\ca/\Phi$ denote the projection functor. If a functor $F:\ca\to\cb$ satisfies the condition $F\circ\Phi\cong F$, then there is a functor $\bar{F}:\ca/\Phi\to \cb$ such that $F=\bar{F}\circ\pi$.

Assume the following three conditions:
\begin{itemize}
\item[(1)] $\ca$ is idempotent complete,
\item[(2)] indecomposable objects in $\ca$ have local endomorphism algebras,
\item[(3)] any indecomposable object $X$ is not isomorphic in $\ca$ to $\Phi^p X$ for any $p\neq 0$.
\end{itemize}
It follows that $\cb$ is also idempotent complete and an object is indecomposable in $\ca$ if and only if it is indecomposable in $\ca/\Phi$. Moreover an indecomposable object in $\ca/\Phi$ also has local endomorphism algebra.

One checks that for two indecomposable objects $X$ and $Y$
\begin{itemize}
\item[(a)] $X$ and $Y$ are isomorphic in $\ca/\Phi$ if and only if  $X$ is isomorphic in $\ca$ to $\Phi^p Y$ for some $p\in\mathbb{Z}$,
\item[(b)] $\rad_{\ca/\Phi}(X,Y)=\bigoplus_{p\in\mathbb{Z}}\rad_\ca(X,\Phi^p Y)$, and $\rad^2_{\ca/\Phi}(X,Y)=\bigoplus_{p\in\mathbb{Z}}\rad^2_\ca(X,\Phi^pY)$.
\end{itemize}
The auto-equivalence $\Phi$ on $\ca$ induces an automorphism $\varphi$ on the Gabriel quiver $\Gamma_\ca$ of $\ca$. It follows from (a) and (b) above that the Gabriel quiver $\Gamma_{\ca/\Phi}$ of $\ca/\Phi$ is exactly the orbit quiver of $\Gamma_\ca$ with respect to $\varphi$.

\smallskip

Assume that $\ca$ is triangulated and $\Phi$ is a triangle equivalence. The orbit category $\ca/\Phi$ is in general not triangulated. We say that a triangulated category $\cb$ is a \emph{triangulated hull} of $\ca/\Phi$ if there is a triangle functor $F:\ca\to\cb$ such that $F\circ\Phi\cong F$, the induced functor $\bar{F}:\ca/\Phi\to \cb$ is fully faithful and $\cb=\thick(\im \bar{F})$. This definition is weaker than the one in \cite{Keller05}.

Assume now that $\ca$ is a triangulated category and $\Phi$ is a triangle equivalence, that both $\ca$ and $\ca/\Phi$ are Hom-finite Krull--Schmidt and that $\ca/\Phi$ is triangulated such that $\pi:\ca\to\ca/\Phi$ is a triangle functor. If $\ca$ has Auslander--Reiten triangles, then so does $\ca/\Phi$ and $\pi$ preserves Auslander--Reiten triangles.

\subsection{The large orbit category}\label{ss:large-orbit-category}

Let $\ca$ be a category with infinite direct sums and $\Phi:\ca\rightarrow\ca$ be an
auto-equivalence. We define the \emph{large orbit category} $\ca\dslash\Phi$: it has the same objects as $\ca$, and the
morphism space from $X$ to $Y$ is defined by
\[\Hom_{\ca\dslash\Phi}(X,Y)
=\Hom_{\ca}(X,\bigoplus_{p\in\mathbb{Z}}\Phi^p Y).\]
Let $f\in\Hom_{\ca\dslash\Phi}(X,Y)$ and $g\in\Hom_{\ca\dslash\Phi}(Y,Z)$. For each $p\in\mathbb{Z}$, we have a morphism $\Phi^p(g):\Phi^p Y\to\Phi^p (\bigoplus_{q\in\mathbb{Z}}\Phi^q Z)=\bigoplus_{q\in\mathbb{Z}}\Phi^q Z$. Let $\tilde{g}\in\Hom_\ca(\bigoplus_{p\in\mathbb{Z}}\Phi^p Y,\bigoplus_{q\in\mathbb{Z}}\Phi^q Z)=\prod_{p\in\mathbb{Z}}\Hom_\ca(\Phi^p Y,\bigoplus_{q\in\mathbb{Z}}\Phi^q Z)$ be the morphism whose $p$-th component is $\Phi^p(g)$. The composition $g\circ f$ of $g$ with $f$ in $\ca\dslash\Phi$ is the composition $\tilde{g}\circ f$ in $\ca$.

\subsection{The universal property}\label{ss:universal-property}
Let $\ca$ be a category with infinite direct sums and $\Phi:\ca\rightarrow\ca$ be an
auto-equivalence.
Let $\pi:\ca\rightarrow\ca\dslash\Phi$ denote the canonical projection functor. It satisfies the following universal property.

\begin{lemma}\label{lem:universal-property}
Let $\cb$ be a category with infinite direct sums. If $F:\ca\rightarrow\cb$ is a functor which commutes with all direct sums such that $F\circ\Phi\cong F$, then there is a functor $\bar{F}:\ca\dslash\Phi\rightarrow\cb$ such that $F=\bar{F}\circ \pi$.
\end{lemma}
\begin{proof}
For $X\in\ca$, define $\bar{F}(X)=F(X)$. Let $\mu: F\circ\Phi\to F$ be the natural isomorphism.
For $f\in\Hom_{\ca\dslash\Phi}(X,Y)=\Hom_\ca(X,\bigoplus_{p\in\mathbb{Z}}\Phi^p Y)$, define $\bar{F}(f)$ as the following composition:
\[
\xymatrix@R=0.8pc@C=2pc{
F(X)\ar[r]^(0.25){F(f)} & F(\bigoplus_{p\in\mathbb{Z}}\Phi^p Y)=\bigoplus_{p\in\mathbb{Z}}F(\Phi^p Y)\ar[r]^(0.65){\oplus_p \mu_p}&\bigoplus_{p\in\mathbb{Z}} F(Y) \ar[r]^(0.55){\mathrm{diag}} & F(Y),
}
\]
where $\mu_p$ is the composition
\[
\xymatrix{
F(\Phi^{p}Y)\ar[r]^{\mu_{\Phi^{p-1}Y}}
 & F(\Phi^{p-1}Y)
 \ar[r] &\ldots\ar[r] & F\Phi (Y)\ar[r]^{\mu_Y} & F(Y).
}
\]
\end{proof}

We say that a functor $F:\ca\to\cb$ commuting with all direct sums is \emph{$\Phi$-orbitally fully faithful} if $F\circ\Phi\cong F$ and the induced functor $\bar{F}:\ca\dslash\Phi\to \cb$ is fully faithful.

\subsection{Compact objects}\label{ss:compact-objects-for-orbit-cat}
Let $\ca$ be a category with infinite direct sums and $\Phi:\ca\rightarrow\ca$ be an
auto-equivalence.
It is easy to see that $\ca\dslash\Phi$ has infinite direct sums and the projection functor $\pi:\ca\to\ca\dslash\Phi$ commutes with all direct sums.

\begin{lemma}\label{lem:compact-objects-in-orbit-cat}
Let $X$ be an object of $\ca$. Then $X$ is compact in $\ca\dslash\Phi$ if and only if $X$ is compact in $\ca$.
\end{lemma}
\begin{proof}Let $\{Y_i|i\in I\}$ be a set of objects  of $\ca$.

Assume that $X$ is compact in $\ca$. We have
\begin{align*}
\Hom_{\ca\dslash\Phi}(X,\bigoplus_{i\in I} Y_i)&=\Hom_{\ca}(X,\bigoplus_{p\in\mathbb{Z}}\Phi^p\bigoplus_{i\in I}Y_i)\\
&=\Hom_{\ca}(X,\bigoplus_{i\in I}\bigoplus_{p\in\mathbb{Z}}\Phi^p Y_i)\\
&=\bigoplus_{i\in I}\Hom_{\ca}(X,\bigoplus_{p\in\mathbb{Z}}\Phi^p Y_i)\\
&=\bigoplus_{i\in I}\Hom_{\ca\dslash\Phi}(X,Y_i),
\end{align*}
showing that $X$ is compact in $\ca\dslash\Phi$.

Assume that $X$ is compact in $\ca\dslash\Phi$. We have an injective map
\begin{align*}
\bigoplus_{i\in I}\Hom_\ca(X,Y_i)&\hookrightarrow \bigoplus_{i\in I}\Hom_\ca(X,\bigoplus_{p\in\mathbb{Z}}\Phi^p Y_i)\\
&=\bigoplus_{i\in I} \Hom_{\ca\dslash\Phi}(X,Y_i)\\
&=\Hom_{\ca\dslash\Phi}(X,\bigoplus_{i\in I}Y_i)\\
&=\Hom_\ca(X,\bigoplus_{p\in\mathbb{Z}}\Phi^p \bigoplus_{i\in I}Y_i),
\end{align*}
whose image is identified with $\Hom_\ca(X,\bigoplus_{i\in I}Y_i)$. This shows that $X$ is compact in $\ca$.
\end{proof}

Let $\ca^c$ be the full subcategory of $\ca$ consisting of compact objects. Then $\Phi$ restricts to an auto-equivalence $\phi:\ca^c\to\ca^c$.  The embedding $\ca^c\to \ca$ induces a fully faithful functor $\ca^c/\phi\to\ca\dslash\Phi$, which, by the preceding lemma, identifies $\ca^c/\phi$ with $(\ca\dslash\Phi)^c$.

\subsection{Orbit categories of triangulated categories}

Let $\cc$ and $\cd$ be triangulated categories with infinite direct sums and let $\Phi:\cc\to\cc$ be a triangle equivalence of $\cc$.

\begin{lemma}\label{lem:compact-generator-of-orbit-category}
Let $F:\cc\to\cd$ be a $\Phi$-orbitally fully faithful triangle functor.
Assume that $\cc$ has a set of compact generators $\{\Phi^pT|p\in\mathbb{Z}\}$ and that $F(T)$ is a compact generator in $\cd$. Let $X$ be an object of $\cc$. Then $\{\Phi^pX|p\in\mathbb{Z}\}$ is a set of compact generators of $\cc$ if and only if $F(X)$ is a compact generator of $\cd$.
\end{lemma}

\begin{proof}
By Theorem~\ref{thm:compact-objects}, $\cc^c=\thick_\cc(\Phi^p T|p\in\mathbb{Z})$ and $\cd^c=\thick_\cd(F(T))$. Moreover, $\{\Phi^p X|p\in\mathbb{Z}\}$ is a set of compact generators of $\cc$ if and only if $\thick_\cc(\Phi^p X|p\in\mathbb{Z})=\cc^c$, and $F(X)$ is a compact generator of $\cd$ if and only if $\thick_\cd(F(X))=\cd^c$.

Assume that $\{\Phi^p X|p\in\mathbb{Z}\}$ is a set of compact generators of $\cc$. Then $F(X)\in\thick_\cd(F(T))$ and $F(T)\in\thick_\cd(F(X))$. So $F(X)$ is a compact generator of $\cd$.

Assume that $F(X)$ is a compact generator of $\cd$. Then $X$ is compact in $\cc\dslash\Phi$. So $\Phi^p X$, by Lemma~\ref{lem:compact-objects-in-orbit-cat}, are compact in $\cc$ for any $p\in\mathbb{Z}$. Further, let $Y$ be an object of $\cc$ such that $\Hom_\cc(\Phi^p X,\Sigma^q Y)=0$ for all $p,q\in\mathbb{Z}$. Then $\Hom_{\cc\dslash\Phi}(X,\Sigma^q Y)=\bigoplus_{p\in\mathbb{Z}}\Hom_\cc(X,\Phi^p \Sigma^q Y)=0$ for all $q\in\mathbb{Z}$, \ie $\Hom_{\cd}(F(X),\Sigma^q F(Y))=0$ for all $q\in\mathbb{Z}$. So $F(Y)=0$. It follows that $Y\cong 0$ in $\cc\dslash\Phi$ and hence $Y\cong 0$ in $\cc$. This shows that $\{\Phi^p X|p\in\mathbb{Z}\}$ is a set of compact generators of $\cc$.
\end{proof}

\subsection{Orbit categories of dg categories}

Let $\ca$ be a dg category with infinite direct sums and $\Phi$ be a dg auto-equivalence. Then $\ca\dslash\Phi$ is also a dg category and the projection functor $\pi:\ca\rightarrow\ca\dslash\Phi$ is a dg functor. Moreover, $H^0\Phi$ is an auto-equivalence of $H^0\ca$ and $H^0(\ca\dslash\Phi)=H^0\ca\dslash H^0\Phi$.
The projection functor $\pi$ induces the projection functor $H^0\pi:H^0\ca\rightarrow H^0\ca\dslash H^0\Phi$. In Lemma~\ref{lem:universal-property}, if $F$ is a dg functor, then $\bar{F}$ is also a dg functor.

Assume that  $\ca$ is pre-triangulated, then $H^0\ca$ is a triangulated category and $H^0\Phi$ is a triangle auto-equivalence of $H^0\ca$.

\begin{lemma}\label{lem:orbit-functor-from-dg-to-triangulated} Let $\cb$ be a pre-triangulated dg categories with infinite direct sums.
If $F:\ca\to\cb$ be a $\Phi$-orbitally fully faithful dg functor, then $H^0F:H^0\ca\to H^0\cb$ is an  $H^0\Phi$-orbitally fully faithful triangle functor.
\end{lemma}

\section{Bigraded dg algebras}\label{s:bigraded-dg-alg}

The notion of group-graded dg algebras was introduced in~\cite{Marcus03}.
In this paper, we will call $\mathbb{Z}$-graded dg algebras \emph{graded dg algebras} and call $\mathbb{Z}\times\mathbb{Z}$-graded dg algebras \emph{bigraded dg algebras}.
Results in this section will be used in Section~\ref{s:derived-equi-graded-alg}.

\subsection{Graded dg algebras}\label{ss:graded-dg-alg}

By a \emph{graded dg algebra} we mean a dg algebra with an extra
grading (often called the Adams grading). More precisely, a graded dg
algebra is a bigraded algebra endowed with a differential $d$ of
bidegree $(1,0)$ such that the following bigraded Leibniz rule holds
\[d(aa')=d(a)a'+(-1)^{|a|_1}ad(a'),\]
where $a$ is homogeneous of bidegree $(|a|_1,|a|_2)$.

Let $A$ be a graded dg algebra. We can construct from $A$ two
(ordinary) dg algebras. One, denoted by $FA$, is obtained from $A$ by
forgetting the Adams grading:
$(FA)^n=\bigoplus_{j\in\mathbb{Z}}A^{n,j}$. This is a dg algebra with differential and multiplication inherited from $A$. The other one, denoted by $\Tot A$, is obtained from $A$ by
taking the total complex:
\[
(\Tot A)^n=\bigoplus_{i,j:i+j=n}A^{ij} \text{ and } d_{\Tot A}(a)=(-1)^j d(a)
\]
for $a$ homogeneous of bidegree $(i,j)$. However, in order to
make it into a dg algebra, we need to introduce a twisted
multiplication:
\[a*a'=(-1)^{ij'}aa'\]
for $a$ and $a'$ homogeneous  of bidegree $(i,j)$ and $(i',j')$, respectively.

\begin{lemma} With the twisted multiplication $*$ the total complex
$\Tot A$ is a dg algebra.
\end{lemma}
 When we say that $\Tot A$ is a dg algebra,
we always mean $\Tot A$ equipped with this twisted multiplication. If $A$ is a graded algebra, viewed as a graded dg algebra concentrated in degrees $(0,*)$, the dg algebra $\Tot A$ is exactly $A$ viewed as a dg algebra with trivial differential.

\begin{proof} Let $a$ and $a'$ be homogeneous elements of $A$
 of bidegree $(i,j)$ and $(i',j')$, respectively. Then the graded
Leibniz rule holds:
\begin{align*}
d_{\Tot A}(a*a')&=(-1)^{ij'}d_{\Tot A}(aa')\\
&=(-1)^{ij'}(-1)^{j+j'}d(aa')\\
&=(-1)^j (-1)^{(i+1)j'}d(a)a'+ (-1)^{j+j'+i}(-1)^{ij'}ad(a')\\
&=(-1)^j d(a)*a'+(-1)^{i+j}(-1)^{j'}a*d(a')\\
&=d_{\Tot A}(a)*a'+(-1)^{i+j}a*d_{\Tot A}(a').
\end{align*}
\end{proof}

\smallskip
Let $H^{*,*}A=\bigoplus_{i,j\in\mathbb{Z}}H^{i,j}A$ be the total cohomology of $A$, which is naturally a
bigraded algebra. We view it as a graded dg algebra with trivial differential.

\begin{lemma}\label{l:cohomology-of-graded-dg-alg}
We have $FH^{*,*}A=H^*FA$ and $\Tot H^{*,*}A=H^*\Tot A$.
\end{lemma}

The two operations $F$ and $\Tot$ extend naturally to functors from
the category of graded dg algebras to the category of dg algebras.
Both functors detect quasi-isomorphisms. The graded dg algebra $A$ is said to be \emph{formal} if $A$ and $H^{*,*}A$ are related by a zigzag of quasi-isomorphisms of graded dg algebras.

\begin{lemma}\label{l:transfering-formalness-graded}
\begin{itemize}
\item[(a)] If $A$ is formal, then both $FA$ and $\Tot A$ are formal.
\item[(b)] If $FA$ has cohomology concentrated in degree $0$, then $A$
is formal.
\item[(c)] If $\Tot A$ has cohomology concentrated in degree $0$,
then $A$ is formal.
\end{itemize}
\end{lemma}

\begin{proof} (a) is clear in view of
Lemma~\ref{l:cohomology-of-graded-dg-alg}.

(b) For $j\in\mathbb{Z}$, let
$A^{*,j}=\bigoplus_{i\in\mathbb{Z}}A^{ij}$. Then $A^{*,j}$ is a
complex and as a complex $FA=\bigoplus_{j\in\mathbb{Z}}A^{*,j}$.
Define $\sigma^{\leq 0}A$ as the bigraded subspace of $A$ such that
$(\sigma^{\leq 0}A)^{*,j}=\sigma^{\leq 0}(A^{*,j})$, the standard truncation of $A^{*,j}$ at degree $0$. It is easy to check that
$\sigma^{\leq 0}A$ is a graded dg subalgebra of $A$. Moreover, that $FA$ has
cohomology concentrated in degree $0$ implies that $A$ has
cohomology concentrated in bidegrees $(0,*)$, which implies that the inclusion $\sigma^{\leq 0}A\hookrightarrow A$ is a quasi-isomorphism. Further, the canonical
projection $\sigma^{\leq 0}A\rightarrow H^{*,*}(\sigma^{\leq 0}A)=H^{*,*}A$ is also a quasi-isomorphism.

(c) Similar to (b). As a complex $\Tot A=\bigoplus_{j\in\mathbb{Z}}A^{*,j}[-j]$. Let $A'$ be the bigraded subspace of $A$ such that $(A')^{*,j}=\sigma^{\leq -j}(A^{*,j})$. Then $A'$ is a graded dg subalgebra of $A$. That $\Tot A$ has cohomology concentrated in degree $0$ means exactly that the
cohomology $H^{*,*}(A)$ of $A$ is concentrated in degrees $(i,-i)$, $i\in\mathbb{Z}$. Therefore the inclusion $A'\hookrightarrow A$ is a quasi-isomorphism. Further, the canonical projection $A'\rightarrow H^{*,*}(A')=H^{*,*}(A)$ is a also a quasi-isomorphism.
\end{proof}

\begin{remark}
The idea of proving formality by considering a second grading
has been used for a long time in geometry,
for example see~\cite{DeligneGriffithsMorganSullivan75}\footnote{We thank Wolfgang Soergel for this reference.}.
\end{remark}

\subsection{Bigraded dg algebras}\label{ss:bigraded-dg-alg}
A \emph{bigraded dg algebra} is a dg algebra with two extra
gradings. Precisely, a bigraded dg algebra is a trigraded algebra
endowed with a differential $d$ of tridegree $(1,0,0)$ such that the
following trigraded Leibniz rule holds
\[d(aa')=d(a)a'+(-1)^{|a|_1}ad(a'),\]
where $a$ is homogeneous of tridegree $(|a|_1,|a|_2,|a|_3)$. We will
refer to the three gradings respectively as the complex grading, the first Adams grading and the second Adams grading.
Let $F_1A$ denote the graded dg algebra obtained from $A$ by forgetting the first Adams grading, and let $F_2A$ denote the graded dg algebra obtained from $A$ by forgetting the second Adams grading.
Let $H^{*,*,*}A=\bigoplus_{i,j,l\in\mathbb{Z}}H^{i,j,l}A$ be the total cohomology of $A$. It is naturally a trigraded algebra. We view it as a bigraded dg algebra with trivial differential.

\begin{lemma}\label{l:cohomology-of-bigraded-dg-alg} We have
$\Tot F_1 H^{*,*,*}A=H^*\Tot F_1 A$ and $\Tot F_2 H^{*,*,*}A=H^*\Tot F_2 A$.
\end{lemma}

The bigraded dg algebra $A$ is said to be \emph{formal} if $A$ is related to $H^{*,*,*}A$ by a zigzag of quasi-isomorphisms of bigraded dg algebras.

\begin{lemma}\label{l:transfering-formalness-bigraded}
Let $A$ be a bigraded dg algebra. If $\Tot F_1 A$ has cohomology
concentrated in degree $0$, then $A$ and $\Tot F_2 A$ are formal.
\end{lemma}
\begin{proof} The proof is divided into two steps: (1) if $\Tot F_1
A$ has cohomology concentrated in degree $0$, then $A$ is formal; (2)
if $A$ is formal, then $\Tot F_2 A$ is formal. The proof for step (1)
is similar to that for Lemma~\ref{l:transfering-formalness-graded}
(c) and the proof for step (2) is similar to that for
Lemma~\ref{l:transfering-formalness-graded} (a).
\end{proof}


\section{Graded algebras as dg algebras}\label{s:graded-alg-as-dg-alg}

This section deals with the relation between the derived category
of the abelian category of graded modules over a graded algebra and
the derived category of dg modules over the graded algebra viewed as a dg algebra
with trivial differential. For a graded algebra $A$ we will show that $\per(A)$ is a triangulated hull of the
orbit category $\ch^b(\grproj A)/\Sigma\circ\langle -1\rangle$.

\medskip

Let $A$ be a graded $k$-algebra.

\subsection{The functor $\Tot$ of taking total complexes}\label{ss:tot}

Let $\Grmod A$ (respectively, $\Grproj A$, $\grproj A$, $\grmod_0 A$) denote the category of graded $A$-modules (respectively, graded projective $A$-modules, finitely generated graded projective $A$-modules, finite-dimensional graded $A$-modules). $\Grmod A$ is an abelian category, so we can form the dg category $\cc_{dg}(\Grmod A)$ of complexes of graded $A$-modules, see Section~\ref{ss:derived-cat-of-abelian-cat}. The dg category $\cc_{dg}(\Grmod A)$ has two natural dg automorphisms: the complex shift $[1]$ and the degree
shifting $\langle 1\rangle$. They commute with each other.

Viewing $A$ as a dg algebra with trivial differential, we
consider the category $\cc_{dg}(A)$ of dg $A$-modules, see
Section~\ref{ss:derived-cat-of-dg-alg}.

There is a natural dg functor from $\cc_{dg}(\Grmod A)$ to
$\cc_{dg}(A)$ constructed as follows. Any object of $\cc_{dg}(\Grmod
A)$ is of the form $M=\bigoplus_{i,j\in\mathbb{Z}} M^{ij}$, where
$M^{i,*}=\bigoplus_{j\in\mathbb{Z}}M^{i,j}$ is a graded $A$-module,
and $d:M^{i,*}\rightarrow M^{i+1,*}$ is a differential (\ie $d$ is
of bidegree $(1,0)$). We can view $M$ as a bicomplex and take the
total complex $\Tot M$
\[(\Tot M)^n=\bigoplus_{i,j:i+j=n}M^{ij}\text{ and } d_{\Tot M}(m)=(-1)^j d_M(m) \text{ for } m\in M^{ij}.\]
It is easy to check that $\Tot M$, with the twisted $A$-action
\[
m*a=(-1)^{ij'}ma
\text{ for } m\in M^{ij}\text{ and }a\in A^{j'},
\]
becomes a dg $A$-module. For a homogeneous
morphism $f\in\Hom_{\cc_{dg}(\Grmod A)}(M,N)$ of degree $p$, we define
$\Tot f\in\Hom_{\cc_{dg}(A)}(\Tot M,\Tot N)$ so that on each component $M^{ij}$ it coincides with $(-1)^{pi}f$.
In this way, we have defined a dg functor
\[\Tot:\cc_{dg}(\Grmod A)\longrightarrow \cc_{dg}(A).\]
It induces a triangle functor
\[\Tot:\ch(\Grmod A)\longrightarrow\ch(A).\]
Since $\Tot$ preserves acyclicity, this in turn induces a triangle functor
\[\Tot:\cd(\Grmod A)\longrightarrow\cd(A).\]


It is easy to check that we have an equality
$\Tot\circ [1]\circ\langle
-1\rangle=\Tot$ of dg endofunctors of $\cc_{dg}(\Grmod A)$. Moreover, $\Tot$ commutes with infinite direct sums. Thus we have an induced dg functor
\[
\overline{\Tot}:\cc_{dg}(\Grmod A)\dslash[1]\circ\langle-1\rangle\longrightarrow\cc_{dg}(A).
\]
On objects $\overline{\Tot}M=\Tot M$ and on
morphisms $\overline{\Tot}$ is defined by
\[\xymatrix{\Hom_{\cc_{dg}(\Grmod A)}(M,
\bigoplus_{p\in\mathbb{Z}}N[p]\langle-p\rangle)\ar[r]^(0.47){\Tot}\ar[ddr]_{\overline{\Tot}}&
\Hom_{\cc_{dg}(A)}(\Tot M, \bigoplus_{p\in\mathbb{Z}}\Tot N[p]\langle-p\rangle)\ar@{=}[d]\\
&\Hom_{\cc_{dg}(A)}(\Tot M,\bigoplus_{p\in\mathbb{Z}}\Tot N)\ar[d]^{\text{diag}_*}\\
&\Hom_{\cc_{dg}(A)}(\Tot M,\Tot N).
}\]
See also \cite[Lemmas 3.2 and 3.4]{ChenYang15} for
the statement (d) of the following corollary.

\begin{theorem}\label{thm:tot-is-orbital}
$\Tot:\cc_{dg}(\Grmod A)\to \cc_{dg}(A)$ is a $([1]\circ\langle-1\rangle)$-orbitally fully faithful dg functor.  We have the following consequences.
\begin{itemize}
\item[(a)] $\Tot:\ch(\Grmod A)\to \ch(A)$ is $(\Sigma\circ\langle-1\rangle)$-orbitally fully faithful.
\item[(b)] $\Tot:\cd(\Grmod A)\to \cd(A)$ is $(\Sigma\circ\langle-1\rangle)$-orbitally fully faithful.
\item[(c)] $\Tot$ induces a fully faithful functor $\ch^b(\Grproj A)/\Sigma\circ\langle-1\rangle\to \cd(A)$. If $A$ has finite global dimension, $\Tot$ induces a fully faithful functor $\cd^b(\Grmod A)/\Sigma\circ\langle-1\rangle\to \cd(A)$.
\item[(d)] $\per(A)$ is a triangulated hull of the orbit category $\ch^b(\grproj
A)/\Sigma\circ\langle-1\rangle$.
\item[(e)] If $A$ is graded hereditary, then $\Tot$ induces triangle equivalences
\begin{align*}
\cd(A)&\simeq\cd(\Grmod A)\dslash\Sigma\circ\langle-1\rangle\\
&\simeq \cd^b(\Grmod A)/\Sigma\circ\langle-1\rangle,\\
\per(A)&\simeq\ch^b(\grproj
 A)/\Sigma\circ\langle-1\rangle,\\
&\hspace{4pt}\simeq \cd^b(\grmod A)/\Sigma\circ\langle-1\rangle,\\
\cd_{fd}(A)&\simeq\cd^b(\grmod_0 A)/\Sigma\circ\langle-1\rangle,
\end{align*}
where $\grmod A$ is the category of finitely presented graded $A$-modules.
\end{itemize}
\end{theorem}
We remark that on these orbit categories $\Sigma\circ\langle-1\rangle$ is the identity on objects but as a functor it is not isomorphic to the identity, see \cite{Keller05cor}.

\begin{proof} For a graded $A$-module $X$ and $i\in\mathbb{Z}$, we define a new graded $A$-module $X_{\diamond i}$ by twisting the $A$-action on $X$:
\[
x\diamond a=(-1)^{ij}xa \text{ for } a\in A^j.
\]
In general $X'$ is not isomorphic to $X$, but there is an identification $\Hom_{\Grmod A}(X,Y)=\Hom_{\Grmod A}(X',Y')$.

For an object $M$ of $\cc_{dg}(\Grmod A)$, notice that
as a graded $A$-module $\Tot M$ is precisely the direct sum
$\bigoplus_{i\in\mathbb{Z}}(M^{i,*})_{\diamond i}\langle-i\rangle$. Therefore for
$M$ and $N$ objects in $\cc_{dg}(\Grmod A)$ we have
\begin{align*}
\Hom_{\cc_{dg}(A)}^n(\Tot M,\Tot N)&=\cHom_A^n(\Tot M,\Tot N)\\
&=\Hom_{\Grmod
A}(\Tot M,(\Tot N)\langle
n\rangle)\\
&=\Hom_{\Grmod A}(\bigoplus_{i\in\mathbb{Z}}
(M^{i,*})_{\diamond i}\langle-i\rangle,\bigoplus_{j\in\mathbb{Z}}
(N^{j,*})_{\diamond j}\langle-j\rangle\langle n\rangle)\\
&=\prod_{i\in\mathbb{Z}}\Hom_{\Grmod
A}((M^{i,*})_{\diamond i}\langle-i\rangle,\bigoplus_{j\in\mathbb{Z}}(N^{j,*})_{\diamond j}\langle n-j\rangle)\\
&=\prod_{i\in\mathbb{Z}}\Hom_{\Grmod
A}(M^{i,*}\langle-i\rangle,\bigoplus_{j\in\mathbb{Z}}N^{j,*}\langle n-j\rangle)\\
&=\prod_{i\in\mathbb{Z}}\Hom_{\Grmod
A}(M^{i,*},\bigoplus_{j\in\mathbb{Z}}N^{j,*}\langle i+n-j\rangle).
\end{align*}
On the other hand,
\begin{align*}
\Hom_{\cc_{dg}(\Grmod
A)}^n(M,\bigoplus_{p\in\mathbb{Z}}N[p]\langle-p\rangle)
&=\Hom_{\Grmod
A}(\bigoplus_{i\in\mathbb{Z}}M^{i,*},\bigoplus_{p\in\mathbb{Z}}(N[p]\langle-p\rangle)^{i+n,*})\\
&=\prod_{i\in\mathbb{Z}}\Hom_{\Grmod
A}(M^{i,*},\bigoplus_{p\in\mathbb{Z}}(N[p]\langle-p\rangle)^{i+n,*})\\
&=\prod_{i\in\mathbb{Z}}\Hom_{\Grmod
A}(M^{i,*},\bigoplus_{p\in\mathbb{Z}}N^{i+n+p,*}\langle-p\rangle)\\
&\hspace{-13pt}\stackrel{j=i+n+p}{=}\prod_{i\in\mathbb{Z}}\Hom_{\Grmod
A}(M^{i,*},\bigoplus_{j\in\mathbb{Z}}N^{j,*}\langle i+n-j\rangle).
\end{align*}
This shows that $\overline{\Tot}:\cc_{dg}(\Grmod A)\dslash[1]\circ\langle-1\rangle\to\cc_{dg}(A)$ is fully faithful. Therefore, $\Tot:\cc_{dg}(\Grmod A)\to\cc_{dg}(A)$ is $([1]\circ\langle-1\rangle)$-orbitally fully faithful, as desired.

(a) This follows from Lemma~\ref{lem:orbit-functor-from-dg-to-triangulated}.

(b) It is easy to see that $\Tot$ sends cofibrant objects in $\cc_{dg}(\Grmod A)$ to cofibrant dg $A$-modules and preserves homotopy equivalences. Therefore, it restricts to a $([1]\circ\langle-1\rangle)$-orbitally fully faithful dg functor $\cc_{dg}^{hp}(\Grmod A)\to \cc_{dg}^{hp}(A)$. Here $\cc_{dg}^{hp}(\Grmod A)$ is the dg subcategory of $\cc_{dg}(\Grmod A)$ of $\ch$-projective objects and $\cc_{dg}^{hp}(A)$ is the dg subcategory of $\cc_{dg}(A)$ of $\ch$-projective dg $A$-modules. They are dg enhancements of $\cd(\Grmod A)$ and $\cd(A)$, respectively. Thus it follows from Lemma~\ref{lem:orbit-functor-from-dg-to-triangulated} that taking $H^0$ yields (b).

(c) We view $\ch^b(\Grproj A)$ as a triangulated subcategory of $\cd(\Grmod A)$. For objects $M,N\in\ch^b(\Grproj A)$, the space $\Hom_{\ch^b(\Grproj A)}(M,\Sigma^p N\langle-p\rangle)$ vanishes for almost all $p\in\mathbb{Z}$. Therefore the embedding $\ch^b(\Grproj A)\to \cd(\Grmod A)$ induces a fully faithful functor $\ch^b(\Grproj A)/\Sigma\circ\langle-1\rangle\to\cd(\Grmod A)\dslash\Sigma\circ\langle-1\rangle$. Now the first statement follows from (b). If $A$ has finite global dimension, then $\cd^b(\Grmod A)=\ch^b(\Grproj A)$ as a triangulated subcategories of $\cd(\Grmod A)$, which implies the second statement.

(d) It follows from (c) that the induced functor $\ch^b(\grproj A)/\Sigma\circ\langle-1\rangle\to\per(A)$ is fully faithful. The image contains $A$ and hence generates $\per(A)$.

(e) Suppose that $A$ is graded hereditary. Consider the following commutative diagram of fully faithful triangle functors
\[
\xymatrix{
\cd(\Grmod A)\dslash\Sigma\circ\langle-1\rangle \ar[r]^(0.7){\overline{\Tot}}& \cd(A).\\
\cd^b(\Grmod A)/\Sigma\circ\langle-1\rangle\ar[u]\ar[ur]
}
\]
By~\cite[Theorem 3.1]{KellerYangZhou09}, the functor $H^*:\cd(A)\rightarrow\Grmod A$ of taking total cohomology induces a bijection on the sets of isomorphism classes of objects. It follows that any object in $\cd(A)$ is isomorphic to its total cohomology, which is viewed as dg $A$-module with trivial differential. Thus all the above three functors are dense, and hence are equivalences. The equivalence $\overline{\Tot}:\cd(\Grmod A)\dslash\Sigma\circ\langle-1\rangle\to\cd(A)$ restricts to an equivalence $\ch^b(\grproj A)/\Sigma\circ\langle-1\rangle\to\per(A)$ on the full subcategories of compact objects (Lemma~\ref{lem:compact-objects-in-orbit-cat} and Theorem~\ref{t:compact-generator}). Since by definition every dg module in $\cd_{fd}(A)$ has finite-dimensional total cohomology, it follows that the restriction $\cd^b(\grmod_0 A)/\Sigma\circ\langle-1\rangle\to\cd_{fd}(A)$ is also an equivalence.

Finally, because $A$ is graded hereditary, the category $\grmod A$ is abelian, and the canonical functor $\ch^b(\grproj A)\to\cd^b(\grmod A)$ is a triangle equivalence. Therefore we have an equivalence $\cd^b(\grmod A)/\Sigma\circ\langle-1\rangle\to\per(A)$.
\end{proof}

\begin{remark}
For bounded complexes $M$ and $N$ in $\cc_{dg}(\Grmod A)$, we actually have
\begin{align*}
\Hom_{\cc_{dg}(\Grmod
A)\dslash\Sigma\circ\langle-1\rangle}(M,N)&=\bigoplus_{p\in\mathbb{Z}}\Hom_{\cc_{dg}(\Grmod
A)}(M,N[p]\langle-p\rangle)\\
&\cong\Hom_{\cc_{dg}(A)}(\Tot M,\Tot N).
\end{align*}
\end{remark}

The orbit category $\ch^b(\grproj
A)/\Sigma\circ\langle-1\rangle$ is in general not triangulated. For
an example, let $A$ be the path algebra of the graded quiver
\[\xymatrix{1\ar@<.7ex>[r]^{\alpha}&2\ar@<.7ex>[l]^{\beta}}\] modulo the ideal generated by
the path $\alpha\beta$, where $\deg(\alpha)=1$ and $\deg(\beta)=0$. We claim that $\ch^b(\grproj A)$ ($\simeq\cd^b(\grmod_0 A)$) is triangle equivalent to $\cd^b(\rep Q)$, where $Q$ is the quiver of type $A^\infty_\infty$ with alternative orientation. In particular, $\ch^b(\grproj A)$ and hence $\ch^b(\grproj
A)/\Sigma\circ\langle-1\rangle$ are discrete. On the contrary, by Proposition~\ref{p:A-L-derived-graded-here}, $A$
is derived equivalent to the path algebra $B$ of the Kronecker
quiver, implying that
$\per(A)\simeq\per(B)$ is tame. The difference between $\ch^b(\grproj A)/\Sigma\circ\langle-1\rangle$
and its triangulated hull $\per(A)$ is surprisingly large. This phenomenon was observed by Amiot--Oppermann in the case of cluster categories,
see~\cite{AmiotOppermann13}.

Let us prove the claim. By covering theory we have $\grmod_0 A\simeq\mod C$, where is $C$ is the quotient of the path category
of the quiver
\[\xymatrix@C=2pc{\ar@{.}[r]&\cdot\ar[r]^{\beta_{i-1}} & 1^{i-1}\ar[r]^{\alpha_{i-1}} & 2^i\ar[r]^{\beta_i}& 1^i\ar[r]^{\alpha_i}&\cdot\ar@{.}[r]&}\]
modulo the ideal generated by $\alpha_i\beta_i$, $i\in\mathbb{Z}$.
Let $T^i_1=\Sigma^{-i-1}S_{1^i}$ and $T^i_2=\Sigma^{-i}P_{2^i}$, where $S_{1^i}$ is the simple module at $1^i$ and $P_{2^i}$ is the projective module at $2^i$. Then $\ct=\{T^i_1,T^i_2|i\in\mathbb{Z}\}$ is a
tilting subcategory of $\cd^b(\grmod_0 A)$. So $\cd^b(\grmod_0 A)\simeq\cd^b(\mod\ct)$. It is easy to see that $\ct$ is isomorphic to the
path category of $Q$, and consequently $\mod\ct\simeq\rep Q$. Therefore $\cd^b(\grmod_0 A)\simeq\cd^b(\rep Q)$.

\medskip
For a bounded complex $M$ of graded $A$-modules, we form the graded dg endomorphism algebra $\mathcal{GE}_{A}(M)$, which is a graded dg algebra, by putting
\[
\mathcal{GE}_{A}(M)^{ij}=\cHom_{\Grmod A}^i(M,M\langle j\rangle).
\]

\begin{lemma}\label{lem:from-graded-dg-endo-alg-to-dg-endo-alg}
There is an isomorphism of dg algebras $\cEnd_A(\Tot M)\cong\Tot\mathcal{GE}_A(M)$.
\end{lemma}
\begin{proof} For a graded dg algebra $\tilde{A}$,
denote by $\Tot' \tilde{A}$ the dg algebra obtained from $\tilde{A}$ by
taking the total complex:
\[
(\Tot' \tilde{A})^n=\bigoplus_{i,j:i+j=n}\tilde{A}^{ij} \text{ and } d_{\Tot' \tilde{A}}(a)=(-1)^j d(a)
\]
with a twisted
multiplication:
\[a\star a'=(-1)^{(i+j)j'}aa'\]
for $a$ and $a'$ homogeneous  of bidegree $(i,j)$ and $(i',j')$, respectively. One checks that the morphism $a\mapsto (-1)^{\frac{j(j+1)}{2}}a$ is an isomorphism of dg algebras $\Tot \tilde{A}\cong \Tot'\tilde{A}$.

Now it is straightforward to check that \begin{align*}
\End_{\cc_{dg}(\Grmod
A)\dslash[1]\circ\langle-1\rangle}(M)=\bigoplus_{p\in\mathbb{Z}}\Hom_{\cc_{dg}(\Grmod
A)}(M,M[p]\langle-p\rangle)
\end{align*}
is $\Tot'\mathcal{GE}_A(M)$. It follows that $\cEnd_A(\Tot M)$, being isomorphic to $\End_{\cc_{dg}(\Grmod
A)\dslash\Sigma\circ\langle-1\rangle}(M)$ by Theorem~\ref{thm:tot-is-orbital}, is isomorphic to $\Tot \mathcal{GE}_A(M)$.
\end{proof}


\section{Derived equivalences of graded algebras}\label{s:derived-equi-graded-alg}

In this section we study derived equivalences of graded algebras, viewed as dg algebras with trivial differential.

\subsection{From graded derived equivalence to derived equivalence}
Let $A$ be a graded algebra and let $FA$ denote the algebra obtained from $A$ by forgetting the
grading. We also have the forgetful functor on graded modules
\[F:\Grmod A\longrightarrow \Mod FA.\]

Let $A$ and $B$ be two graded algebras. We say that $A$ and $B$ are \emph{graded equivalent} if there is a decomposition $A=P_1\oplus\ldots\oplus P_n$ and integers $a_1,\ldots,a_n$ such that as a graded algebra $B$ is isomorphic to $\bigoplus_{i\in\mathbb{Z}}\Hom_{\Grmod A}(P_1\langle a_1\rangle\oplus\ldots\oplus P_n\langle a_n\rangle,(P_1\langle a_1\rangle\oplus\ldots\oplus P_n\langle a_n\rangle)\langle i\rangle)$, they are
\emph{graded derived equivalent} if there is a triangle equivalence
$\cd(\Grmod A)\simeq\cd(\Grmod B)$ commuting with the degree
shiftings, and they are \emph{derived equivalent} if $\cd(A)$ and
$\cd(B)$ are triangle equivalent (here $A$ and $B$ are considered as
dg algebras with trivial differential).

Graded derived equivalence admits an analogue of Rickard's theorem \cite{Rickard89,Rickard91}.
An object $T$ in $\cd(\Grmod A)$ is called a \emph{graded tilting complex} if  $\{T\langle i\rangle\}_{i\in\mathbb{Z}}$ is a set of generators of $\ch^b(\grproj A)$ and $\bigoplus_{i\in\mathbb{Z}}\Hom_{\cd(\Grmod A)}(T,\Sigma^n T\langle i\rangle)=0$ unless $n=0$. This definition is different from the one in~\cite[Section 2.2]{Marcus03}. That they are equivalent is hidden in the proof of \cite[Theorem 2.4]{Marcus03}.

\begin{theorem}\label{t:graded-rickard's-theorem}\emph{(\cite[Theorem 2.4]{Marcus03})}
Let $A$ and $B$ be two graded algebras. Then the following statements are equivalent:
\begin{itemize}
\item[(i)] $A$ and $B$ are graded derived equivalent;
\item[(ii)] there is a graded tilting complex $T$ over $A$  whose graded endomorphism algebra $\bigoplus_{i\in\mathbb{Z}}\Hom_{\cd(\Grmod A)}(T,T\langle i\rangle)$ is isomorphic as a graded algebra to $B$;
\item[(iii)] there is a complex of graded $B$-$A$-bimodules $T$ such that the triangle functor $?\lten_B T:\cd(\Grmod B)\rightarrow\cd(\Grmod A)$ is an equivalence.
\item[(iv)]  there is a complex of graded $B$-$A$-bimodules $T$ such that the triangle functor $?\lten_{FB} FT:\cd(FB)\rightarrow\cd(FA)$ is an equivalence.
\end{itemize}
\end{theorem}

The complexes $T$ in (iii) and (iv) can be taken the same complex. It is clear that graded equivalent graded algebras are graded derived equivalent. We remind the reader that there exist triangle equivalences of graded module categories which do not commute with the degree shiftings. See~\cite[Example 8.8]{AmiotOppermann14} for examples.

For a graded algebra $A$, it is known that $\{A\langle i\rangle|i\in\mathbb{Z}\}$ is a set of compact generators of $\cd(\Grmod A)$. Recall from Theorem~\ref{t:compact-generator} that $A$ is a compact generator of $\cd(A)$.

\begin{theorem}\label{thm:der-equiv-on-three-levels}
 Let $A$ and $B$ be two graded algebras, and let $T$ be a complex of graded $B$-$A$-bimodules. The following statements are
equivalent
\begin{itemize}
\item[(i)] $?\lten_B T:\cd(\Grmod B)\rightarrow\cd(\Grmod A)$ is a graded triangle equivalence,
\item[(ii)] $?\lten_B \Tot T:\cd(B)\rightarrow\cd(A)$ is a triangle equivalence,
\item[(iii)] $?\lten_{FB} FT:\cd(FB)\rightarrow\cd(FA)$ is a triangle equivalence.
\end{itemize}
\end{theorem}
\begin{proof} The equivalence between (i) and (iii) is part of Theorem~\ref{t:graded-rickard's-theorem}. We prove that (i) and (ii) are equivalent. By \cite[Lemma 4.2]{Keller94} we have
\begin{eqnarray*}
\lefteqn{?\lten_B T:\cd(\Grmod B)\to \cd(\Grmod A)\text{ is a triangle equivalence}}\\
&\Leftrightarrow& \{\Sigma^pT\langle-p\rangle|p\in\mathbb{Z}\} \text{ is a set of compact generators of } \cd(\Grmod A) \text{ and }  \\
&&\Hom_{\cd(\Grmod B)}(B,\Sigma^i B\langle -p\rangle)\to\Hom_{\cd(\Grmod A)}(T,\Sigma^i T\langle-p\rangle)\text{ is an isomorphism}\\
&&\text{for any } i,p\in\mathbb{Z}\\
&\Leftrightarrow& \Tot T \text{ is a compact generator of } \cd(A) \text{ and } \\ && \Hom_{\cd(B)}(B,\Sigma^i B)\to\Hom_{\cd(A)}(\Tot T,\Sigma^i \Tot T) \text{ is an isomorphism for any }i\in\mathbb{Z}\\
&\Leftrightarrow& ?\lten_B \Tot T:\cd(B)\rightarrow\cd(A) \text{ is a triangle equivalence.}
\end{eqnarray*}
Here the second `$\Leftrightarrow$' uses Lemma~\ref{lem:compact-generator-of-orbit-category} and Theorem~\ref{thm:tot-is-orbital}.
\end{proof}

In view of Theorem~\ref{t:graded-rickard's-theorem} and
Theorem~\ref{thm:der-equiv-on-three-levels} we have an easy and useful corollary.

\begin{corollary}\label{c:grade-derived-equiv-induce-derived-equiv}
Let $A$ and $B$ be two graded algebras. If they are graded derived equivalent, then they are derived equivalent. In particular, if $FA$ and $FB$ are derived equivalent via a gradable tilting complex, then $A$ and $B$ are derived equivalent.
\end{corollary}

\subsection{Grading change}
In this subsection, we study the following question: if two graded algebras $A$ and $B$ are derived equivalent and if we change the grading on $A$ to obtain a new graded algebra $A'$, is there a graded algebra $B'$ obtained from $B$ by changing the grading so that $A'$ and $B'$ are derived equivalent? We will give a partial answer in Proposition~\ref{p:derived-equiv-graded-alg-4}.
\smallskip

Let $A$ be a bigraded algebra, and let $\Grmod A$ denote the category of bigraded $A$-modules. We fix a bounded complex $M$ of finitely generated bigraded
projective $A$-modules, and form the bigraded dg endomorphism algebra
$\mathcal{BE}_A(M)$
\[(\mathcal{BE}_A(M))^{ijl}=\cHom_{\Grmod A}^i(M,M\langle (j,l)\rangle).\]
For $\varepsilon\in\{1,2\}$, $F_\varepsilon M$ is a bounded complex of finitely generated graded projective modules over the graded algebra $F_\varepsilon A$.

\begin{lemma}\label{lem:from-bigraded-dg-endo-alg-to-graded-dg-endo-alg-to-dg-endo-alg}
Let $\varepsilon\in\{1,2\}$. As a graded dg algebra
\[
\mathcal{GE}_{F_\epsilon A}(F_\epsilon M)=F_\varepsilon \mathcal{BE}_A(M)
\]
and as a dg algebra
\[
\cEnd_{F_\varepsilon A}(\Tot F_\varepsilon
M)\cong\Tot\mathcal{GE}_{F_\epsilon A}(F_\epsilon M)=\Tot F_\varepsilon\mathcal{BE}_A(M).
\]
\end{lemma}
\begin{proof}
The first equality is clear from definition, and the second one follows from Lemma~\ref{lem:from-graded-dg-endo-alg-to-dg-endo-alg}.
\end{proof}

We remind the reader that in $\cEnd_{F_\varepsilon A}(\Tot F_\varepsilon
M)$ we view $F_\varepsilon A$ as a dg algebra with trivial differential.

\begin{corollary} We have an isomorphism of graded algebras
\[\bigoplus_{p\in\mathbb{Z}}\Hom_{\cd(F_\varepsilon A)}(\Tot F_\varepsilon M,\Sigma^p \Tot F_\varepsilon M)\cong\Tot F_\varepsilon \bigoplus_{i,j,l\in\mathbb{Z}}\Hom_{\cd(\Grmod A)}(M,\Sigma^i M\langle(j,l)\rangle),\]
where $\varepsilon=1,2$.
\end{corollary}

\begin{proof}
This is a consequence of Lemma~\ref{l:cohomology-of-bigraded-dg-alg} and Lemma~\ref{lem:from-bigraded-dg-endo-alg-to-graded-dg-endo-alg-to-dg-endo-alg} because there is an isomorphism of graded algebras
\[
\bigoplus_{p\in\mathbb{Z}}\Hom_{\cd(F_\varepsilon A)}(\Tot F_\varepsilon M,\Sigma^p \Tot F_\varepsilon M)\cong H^*\cEnd_{F_\varepsilon A}(\Tot F_\varepsilon
M)
\]
and an isomorphism of trigraded algebras
\[
\bigoplus_{i,j,l\in\mathbb{Z}}\Hom_{\cd(\Grmod A)}(M,\Sigma^i M\langle(j,l)\rangle)\cong H^{*,*,*}\mathcal{BE}_A(M).
\]
\end{proof}

\begin{lemma}\label{l:transfering-formalness-bigraded-endo}
If $\cEnd_{F_1 A}(\Tot F_1 M)$
has cohomology concentrated in degree $0$, then $\cEnd_{F_2 A}(\Tot F_2 M)$ is formal.
\end{lemma}
\begin{proof} This follows immediately from Lemma~\ref{lem:from-bigraded-dg-endo-alg-to-graded-dg-endo-alg-to-dg-endo-alg} and Lemma~\ref{l:transfering-formalness-bigraded}.
\end{proof}

\begin{proposition}\label{p:derived-equiv-graded-alg-4} Assume that $\{M\langle(i,j)\rangle|i,j\in\mathbb{Z}\}$ is a set of compact generators of $\cd(\Grmod A)$.
If $\Tot F_1M$ is a tilting object in $\cd(F_1 A)$, then the two graded algebras $F_2 A$ and $H^*\cEnd_{F_2 A}(\Tot F_2 M)$ are derived equivalent. In other words, if the trigraded algebra $B=\bigoplus_{i,j,l\in\mathbb{Z}}\Hom_{\cd(\Grmod A)}(M,\Sigma^i M\langle(j,l)\rangle)$ is concentrated in tridegrees $(i,*,-i)$ ($i\in\mathbb{Z}$), then the two graded algebras $F_2 A$ and $\Tot F_2 B$ are derived equivalent.
\end{proposition}
\begin{proof}
Since $\{M\langle(i,j)\rangle|i,j\in\mathbb{Z}\}$ is a set of compact generators of $\cd(\Grmod A)$, it follows that $\{F_2M\langle i\rangle|i\in\mathbb{Z}\}$ is a set of compact generators of $\cd(\Grmod F_2A)$. So $\{\Sigma^{i}F_2M\langle -i\rangle|i\in\mathbb{Z}\}$ is a set of compact generators of $\cd(\Grmod F_2A)$. Applying Lemma~\ref{lem:compact-generator-of-orbit-category} to the functor $\Tot:\cd(\Grmod F_2A)\to\cd(F_2A)$, we obtain that $\Tot F_2M$ is a compact generator of $\cd(F_2A)$. Now by Proposition~\ref{p:derived-equiv}, $F_2A$ and $\cEnd_{F_2A}(\Tot F_2M)$ are derived equivalent.

That $\Tot F_1M$ is a tilting object in $\cd(F_1A)$ means that $\cEnd_{F_1A}(\Tot F_1M)$ has cohomology concentrated in degree $0$. By Lemma~\ref{l:transfering-formalness-bigraded-endo}, $\cEnd_{F_2A}(\Tot F_2M)$ is formal, and hence by Lemma~\ref{l:formal-dg-alg} it is derived equivalent to $H^*\cEnd_{F_2A}(\Tot F_2M)$. This completes the proof.
\end{proof}

Briefly speaking, Proposition~\ref{p:derived-equiv-graded-alg-4} says the following: if a graded algebra $A$ is derived equivalent to an (ordinary) algebra $B$ via a nice tilting object, and $A'$ is a graded algebra obtained from $A$ by changing the grading, then there is a suitable graded algebra structure $B'$ on $B$ such that $A'$ and $B'$ are derived equivalent.

\section{Gentle one-cycle algebras}\label{s:gentle-one-cycle-alg}

In this section we recall results on gentle algebras of Assem--Skowro\'nski~\cite{AssemSkowronski87},
Vossieck~\cite{Vossieck01} and  Bobi\'{n}ski--Geiss--Skowro\'nski~\cite{BobinskiGeissSkowronski04}.

\subsection{Gentle algebras}\label{ss:gentle-alg}

Let $Q$ be a finite quiver and $I$ a set of minimal relations. We
call the algebra $kQ/(I)$ a \emph{gentle algebra} if the following
conditions hold
\begin{itemize}
\item[(1)]
for each vertex of $Q$ there are at most two incoming arrows and at
most two outgoing arrows,
\item[(2)]
for each arrow $\beta$ of $Q$, both the number of arrows $\alpha$
with $t(\alpha)=s(\beta)$ and $\beta\alpha\notin I$ and the number
of arrows $\gamma$ with $s(\gamma)=t(\beta)$ and $\gamma\beta\notin
I$ are not greater than $1$,
\item[(3)]
for each arrow $\beta$ of $Q$, both the number of arrows $\alpha$
with $t(\alpha)=s(\beta)$ and $\beta\alpha\in I$ and the number of
arrows $\gamma$ with $s(\gamma)=t(\beta)$ and $\gamma\beta\in I$ are
not greater than $1$,
\item[(4)]
all relations in $I$ are paths of length $2$.
\end{itemize}

\begin{example}
\label{ex:gentle-alg} The path algebra of $\xymatrix{1\ar@<.7ex>[r]^{\alpha}&2\ar@<.7ex>[l]^{\beta}\ar[r]^{\gamma} &3}$ modulo the relations $\alpha\beta$ and $\gamma\alpha$ is a typical example of gentle algebras.
\end{example}

The repetitive algebra of a gentle algebra is special biserial (\cite[Proposition 4]{Schroer99}), which is always tame. It follows that the bounded derived category of a gentle algebra is tame, since the bounded derived category of an algebra is triangle equivalent to a full subcategory
of the stable category of its repetitive algebra (\cite[Theorem 4.9]{Happel87}).

\subsection{Gentle one-cycle algebras and derived discrete algebras}\label{ss:derived-discrete-alg}
A \emph{gentle one-cycle algebra} is a gentle algebra whose
underlying graph contains exactly one cycle. It \emph{satisfies the
clock condition} if it has the same number of clockwise and
counterclockwise oriented relations on the cycle. For example, the algebra in Example~\ref{ex:gentle-alg} is a gentle one-cycle algebra which does not satisfy the clock condition.

Two algebras $A$ and $B$ are \emph{tilting-cotilting equivalent} if there is a sequence of algebras $A_0=A,A_1,\ldots,A_s=B$ and $A_i$-modules $T_{i+1}$
such that $\End_{A_i}(T_{i+1})\cong A_{i+1}$ and $T_{i+1}$ is a tilting $A_i$-module of
projective dimension at most $1$ or a cotilting $A_i$-module of injective dimension at most $1$. Gentle algebras are Gorenstein~\cite{GeissReiten05}, so co-tilting modules over gentle algebras are tilting modules. Therefore two gentle algebras which are tilting-cotilting equivalent are derived equivalent.

\begin{theorem}[{\cite[Theorem (A)]{AssemSkowronski87}}]\label{t:gentle-one-cycle-clock}
A gentle one-cycle algebra satisfying the clock condition is tilting-cotilting equivalent to the path algebra of a quiver of type $\widetilde{A}_{p,q}$
for some $p,q>0$.
\end{theorem}

This result was stated in \cite{AssemSkowronski87} under the assumption that $k$ is algebraically closed, but its proof is field-independent.

\smallskip
Let $\Omega=\{(r,n,m)\in\mathbb{Z}^3|n\geq r\geq 1,m\geq 0\}$. In~\cite{BobinskiGeissSkowronski04}, the authors construct a family of gentle one-cycle algebras not
satisfying the clock condition:
$\Lambda(r,n,m)=kQ(r,n,m)/I(r,n,m)$ for $(r,n,m)\in\Omega$, where $Q(r,n,m)$ is the quiver
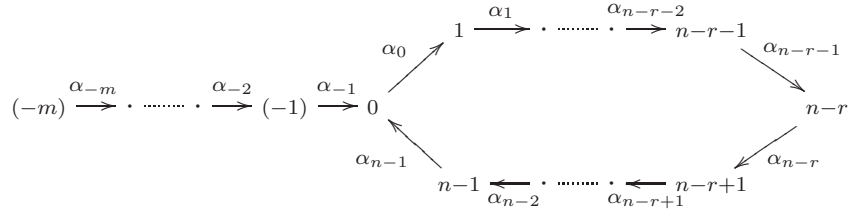
\[
\xymatrix@R=1.2pc@C=1.2pc{&&&&& \scriptstyle{1}\ar[r]^{\scriptstyle{\alpha_1}}&\cdot\ar@{.}[r]&\cdot\ar[r]^(.35){\scriptstyle{\alpha_{n-r-2}}} & \scriptstyle{n-r-1}\ar[rd]^{\scriptstyle{\alpha_{n-r-1}}}&\\
\scriptstyle{(-m)}\ar[r]^(.6){\scriptstyle{\alpha_{-m}}}&\cdot\ar@{.}[r]&\cdot\ar[r]^(.4){\scriptstyle{\alpha_{-2}}}&\scriptstyle{(-1)}\ar[r]^(.6){\scriptstyle{\alpha_{-1}}} & \scriptstyle{0}\ar[ru]^{\scriptstyle{\alpha_0}} &&&& & \scriptstyle{n-r}\ar[ld]^{\scriptstyle{\alpha_{n-r}}}\\
&&&&&\scriptstyle{n-1}\ar[lu]^{\scriptstyle{\alpha_{n-1}}} &
\cdot\ar[l]^(.35){\scriptstyle{\alpha_{n-2}}}&\cdot\ar@{.}[l]
&\scriptstyle{n-r+1}\ar[l]^(.65){\scriptstyle{\alpha_{n-r+1}}}}
\]
and $I(r,n,m)$ is the ideal of $kQ(r,n,m)$ generated by the paths
$\alpha_0\alpha_{n-1}$, $\alpha_{n-1}\alpha_{n-2}$, $\ldots$,
$\alpha_{n-r+1}\alpha_{n-r}$.
Notice that the algebra $\Lambda(r,n,m)$ is of finite global dimension
 if $n>r$ and  is of infinite global dimension if $n=r$.

\begin{theorem}[{\cite[Theorem 2.1]{Vossieck01} and \cite[Theorem A]{BobinskiGeissSkowronski04}}]\label{t:discrete-derived-cat} Assume that $k$ is algebraically closed.
Let $A$ be a finite-dimensional algebra not derived equivalent to a
Dynkin quiver. Then the following are equivalent
\begin{itemize}
\item[(i)] $A$ is derived discrete,
\item[(ii)] $A$ is a gentle one-cycle algebra not satisfying the clock condition,
\item[(iii)] $A$ is derived equivalent to $\Lambda(r,n,m)$ for some $(r,n,m)$,
\item[(iv)] $A$ is tilting-cotilting equivalent to $\Lambda(r,n,m)$ for some $(r,n,m)$.
\end{itemize}
\end{theorem}

We remark that the proof of the implication (ii)$\Rightarrow$(iii) in \cite{BobinskiGeissSkowronski04} is field-independent.

\section{The derived category of a graded $\widetilde{A}_{p,q}$}\label{s:graded-a-pq}

Let $\Pi=\{(p,q,r)\in\mathbb{Z}^3|p\geq 0,q>0\}$. For $(p,q,r)\in\Pi$, let $\Gamma(p,q,r)$ be the complete path algebra of the graded quiver
\[\xymatrix@R=0.7pc@C=0.7pc{&p+q\ar[dl]_{\alpha_{p+q}}&\cdot\ar[l]\ar@{.}[r]&\cdot&p+2\ar[l]&\\
1&&&&&p+1\ar[dl]^{\alpha_p}\ar[ul]_{\alpha_{p+1}},\\
&2\ar[ul]^{\alpha_1}&\cdot\ar[l]^{\alpha_2}\ar@{.}[r]&\cdot&p\ar[l]&}\]
where $\deg(\alpha_i)=\delta_{i,p+q}r$. This is a graded hereditary algebra. In this section we study
$\Grmod\Gamma(p,q,r)$ and  the derived category of $\Gamma(p,q,r)$ which is viewed as a dg algebra
with trivial differential. We start with
recalling some well-known facts.

\subsection{Representations over a quiver of type
$A_\infty^\infty$}\label{ss:representation-a-double-infty}
The descriptions of indecomposable representations and Auslander--Reiten quivers in this subsection can be obtained using results in \cite[Section 5]{BautistaLiuPaquette13}.

Let $Q$ be a quiver of type $A_\infty^\infty$ whose vertices are indexed by $\mathbb{Z}$ and $i$ is adjacent to $i-1$ and $i+1$:
\[\xymatrix{\ldots\ar@{-}[r]&i-1\ar@{-}[r]&i\ar@{-}[r]&i+1\ar@{-}[r]&\ldots}.\]
For $a,b\in\{-\infty\}\cup\mathbb{Z}\cup\{+\infty\}$ with $a>b$, we construct a representation of $Q$ by setting
the $k$-vector space $M_{a,b}(i)$ associated to a vertex $i$ as
\[M_{a,b}(i)=\begin{cases} k & \text{if } b\leq i\leq a-1\\ 0 & \text{otherwise}\end{cases}\]
and setting the $k$-linear map $M_{a,b}(\alpha)$ associated to an arrow $\alpha:i\rightarrow j$ as
\[M_{a,b}(\alpha)=\begin{cases} \mathrm{id}_k &\text{if } b\leq i,j\leq a-1\\ 0 & \text{otherwise}.\end{cases}\]
These representations are indecomposable with endomorphism algebra isomorphic to $k$ and every indecomposable representation is isomorphic to one such. The simple representations are precisely the $M_{i+1,i}$'s ($i\in\mathbb{Z}$) and the finite-dimensional indecomposable representations are precisely the $M_{a,b}$'s with $a,b\in\mathbb{Z}$ and $a>b$.

\subsubsection{The linear orientation}\label{sss:linear}
Let $Q=Q^l$ be the quiver of type $A_\infty^\infty$ with linear orientation:
\[\xymatrix{\ldots\ar[r]&i-1\ar[r]&i\ar[r]&i+1\ar[r]&\ldots}.\]
Let $s$ denote the unique automorphism of $Q$ induced by $i\mapsto i-1$.

The Gabriel quiver of $\Rep Q$ consists of four connected components:
\begin{itemize}
\item[$\cdot$] the $M_{a,b}$, $a,b\in\mathbb{Z},a>b$, form a component $\cx$ of type $\mathbb{Z}A_\infty$,
\item[$\cdot$] the $M_{+\infty,b}$, $b\in\mathbb{Z}$, form a component $\cy$ of type $Q$ (these are the projectives),
\item[$\cdot$]
the $M_{a,-\infty}$, $a\in\mathbb{Z}$, form a component $\cz$ of type $Q$ (these are the injectives),
\item[$\cdot$] the representation $M_{+\infty,-\infty}$ forms a component $\ca$ of type $A_1$.
 \end{itemize}
The first component $\cx$ is the Gabriel quiver of $\rep Q$. This category has Auslander--Reiten sequences, and the Auslander--Reiten translation $\tau=s_*^{-1}$ sends $M_{a,b}$ ($a,b\in\mathbb{Z}$, $a>b$) to $M_{a+1,b+1}$.

\subsubsection{The generalized zigzag orientation}\label{sss:zigzag}
Let $p\geq q$ be positive integers. Let $Q=Q^z$ be the quiver of type $A_\infty^\infty$ with the following generalized zigzag orientation:
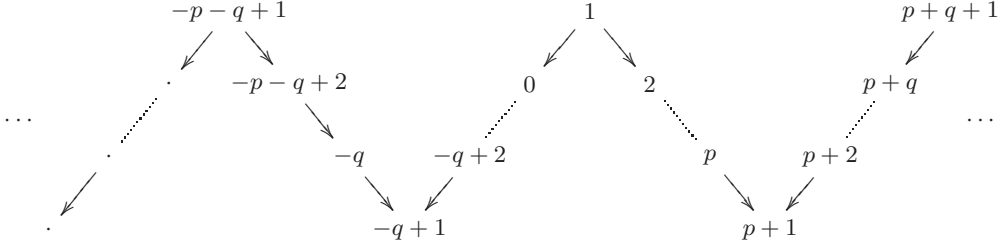
\[
{\scriptsize
\begin{xy} 0;<0.75pt,0pt>:<0pt,-0.55pt>::
(15,75) *+{\cdots}="",
(30,150) *+{\cdot}="0",
(60,100) *+{\cdot}="1",
(90,50) *+{\cdot}="2",
(120,0) *+{-p-q+1}="3",
(150,50) *+{-p-q+2}="4",
(180,100) *+{-q}="5",
(210,150) *+{-q+1}="6",
(240,100) *+{-q+2}="7",
(270,50) *+{0}="8",
(300,0) *+{1}="9",
(330,50) *+{2}="10",
(360,100) *+{p}="11",
(390,150) *+{p+1}="12",
(420,100) *+{p+2}="13",
(450,50) *+{p+q}="14",
(480,0) *+{p+q+1}="15",
(495,75) *+{\cdots}="",
"1", {\ar "0"}, "2", {\ar@{.} "1"}, "3", {\ar "2"},
"3", {\ar "4"}, "4", {\ar "5"}, "5", {\ar "6"},
"7", {\ar "6"}, "8", {\ar@{.} "7"}, "9", {\ar "8"},
"9", {\ar "10"}, "10", {\ar@{.} "11"}, "11", {\ar "12"},
"13", {\ar "12"}, "14", {\ar@{.} "13"}, "15", {\ar "14"},
\end{xy}
}
\]
Let $s$ denote the unique automorphism of $Q$ induced by $i\mapsto i+p+q$.

All indecomposable projective (respectively, injective) representations are finite-dimensional.  The Gabriel quiver of $\Rep Q$ consists of nine connected components:
\begin{itemize}
\item[$\cdot$] the preprojective indecomposable representations form a component $\cp$ of type $\mathbb{N}Q^{op}$,
\item[$\cdot$] the representations $\bigsqcup_{n\in\mathbb{Z}}s_*^n\{M_{p+2,1},M_{1,0},\ldots,M_{-q+3,-q+2}\}$ and the iterated extensions of them form a regular component $\cx^1$ of type $\mathbb{Z}A_\infty$,
\item[$\cdot$] the representations $\bigsqcup_{n\in\mathbb{Z}}s_*^n\{M_{2,-q+1},M_{3,2},\ldots,M_{p+1,p}\}$ and the iterated extensions of them form a regular component $\cx^2$ of type $\mathbb{Z}A_\infty$,
\item[$\cdot$] the $M_{+\infty,b}$, $b\equiv -q+2,\ldots,1~(\mathrm{mod}~p+q)$, form a component $\cy^1$ of type $Q^{l}$,
\item[$\cdot$] the $M_{+\infty,b}$, $b\equiv 2,\ldots,p+1~(\mathrm{mod}~p+q)$, form a component $\cy^2$ of type $Q^{l}$,
\item[$\cdot$] the $M_{a,-\infty}$, $a\equiv 1,\ldots,p~(\mathrm{mod}~p+q)$, form a component $\cz^1$ of type $Q^{l}$,
\item[$\cdot$] the $M_{a,-\infty}$, $a\equiv p+1,\ldots,p+q~(\mathrm{mod}~p+q)$, form a component $\cz^2$ of type $Q^{l}$,
\item[$\cdot$] the preinjective indecomposable representations form a component $\ci$ of type $(-\mathbb{N})Q^{op}$,
\item[$\cdot$] the representation $M_{+\infty,-\infty}$ forms a component $\ca$ of type $A_1$.
\end{itemize}
The components $\cp$, $\cx^1$, $\cx^2$ and $\ci$ form the Gabriel quiver of $\rep Q$. This category has Auslander--Reiten sequences, and on the two regular components $\cx^1$ and $\cx^2$ the Auslander--Reiten translation acts respectively as
\[\tau M_{p+2,1}=M_{1,0},~~\tau
M_{1,0}=M_{0,-1},~~\ldots,~~\tau M_{-q+3,-q+4}=M_{-q+3,-q+2},\]
\[\text{ and }\tau M_{-q+3,-q+2}=M_{-q+2,-p-q+1}=s_*^{-1}M_{p+2,1}.\]
\[\tau M_{2,-q+1}=M_{3,2},~~\tau
M_{3,2}=M_{4,3},~~\ldots,~~\tau M_{p,p-1}=M_{p+1,p},\]
\[\text{ and } \tau
M_{p+1,p}=M_{p+q+1,p+1}=s_*M_{2,-q+1}.\] In
particular, on $\cx^1$ and $\cx^2$ we have $\tau^q=s_*^{-1}$ and $\tau^p=s_*$, respectively.

\subsection{Graded modules over $\Gamma(p,q,r)$}\label{ss:graded-module-over-graded-affine-A}

Fix a triple $(p,q,r)\in\Pi$ and let $\Gamma(p,q,r)$ be the graded
algebra defined in the beginning of this section. If $r=0$, it is a
tame hereditary algebra and $\Grmod\Gamma(p,q,0)$ is the direct sum
of $\mathbb{Z}$ copies of $\Mod\Gamma(p,q,0)$. In the sequel of this
subsection, we assume $r\neq 0$. Let $Q$ be the following quiver
\begin{itemize}
\item[$\cdot$] if $p=0$, $Q$ is the disjoint union of $|r|$ copies of $Q^l$ defined in Section~\ref{sss:linear};
\item[$\cdot$] if $p>0$, $Q$ is the disjoint union of $|r|$ copies of $Q^z$ defined in Section~\ref{sss:zigzag}.
\end{itemize}
The vertices are labeled by $\{(j,i)|0\leq j\leq
|r|-1,i\in\mathbb{Z}\}$. Define $\sigma$ to be the unique
automorphism of $Q$ which takes the following values on vertices:
\[\sigma(j,i)=\begin{cases} (j-1,i) & \text{ if } 1\leq j\leq |r|-1,\\ (|r|-1,s^{\mathrm{sgn}(r)}(i)) & \text{ if } j=0 \text{ and } p\neq 0,\\
(|r|-1,s^{q\cdot\mathrm{sgn}(r)}(i)) & \text{ if } j=0 \text{ and } p=0,
\end{cases}\]
where $s$ was defined in
Section~\ref{ss:representation-a-double-infty}, and
$\mathrm{sgn}(r)$ is the sign of $r$, \ie
\[\mathrm{sgn}(r)=\begin{cases} 1 & \text{ if } r>0;\\ -1 & \text{ if } r<0.\end{cases}\]
In particular, $\sigma^r$ is an automorphism on each connected component and $$\sigma^r=\begin{cases} s & \text{ if } p\neq 0\\ s^q & \text{ if } p=0.\end{cases}$$

\begin{lemma} There is an equivalence $C:\Grmod\Gamma(p,q,r)\simeq
\Rep(Q)$ such that $C\circ\langle 1\rangle=\sigma_*\circ C$, which restricts to an equivalence $\grmod_0\Gamma(p,q,r)\simeq\rep Q$.
\end{lemma}
\begin{proof} This is standard in covering theory. The functor $C$ takes a graded module $M=\bigoplus_{n\in\mathbb{Z}} M^n$ to the representation $V$ with the vector space $V(j,i)$ associated to the vertex $(j,i)$ being $M^{-\lfloor \frac{i-1}{p+q}\rfloor r+j} e_{i-(p+q)\lfloor\frac{i-1}{p+q}\rfloor r}$. Here for a rational number $x$, we denote by $\lfloor x\rfloor$ the greatest integer smaller than or equal to $x$.
\end{proof}

\subsection{The derived category of $\Gamma(p,q,r)$}\label{ss:derived-cat-of-graded-tildeA}

Let $(p,q,r)\in\Pi$. Assume $r\neq 0$.  As $\Gamma(p,q,r)$ is graded hereditary, all indecomposable objects in $\cd(\Grmod \Gamma(p,q,r))$ are shifts of indecomposable graded modules. We have a commutative diagram of triangle functors
\[
\xymatrix{
\cd^b(\Grmod\Gamma(p,q,r))/\Sigma\circ\langle-1\rangle\ar[r]&\cd(\Gamma(p,q,r))\\
\ch^b(\grproj\Gamma(p,q,r))/\Sigma\circ\langle-1\rangle\ar[r]\ar[u]&\per(\Gamma(p,q,r))\ar[u]\\
\cd^b(\grmod_0\Gamma(p,q,r))/\Sigma\circ\langle-1\rangle\ar[r]\ar[u]&\cd_{fd}(\Gamma(p,q,r)),\ar[u]
}
\]
where the horizontal functors are equivalences by Theorem~\ref{thm:tot-is-orbital}(e) and the vertical functors are fully faithful. When $p> 0$, the two lower vertical functors are equalities; when $p=0$, they are not dense.
The Gabriel quiver of $\cd^b(\Grmod \Gamma(p,q,r))$ (respectively, $\ch^b(\grproj\Gamma(p,q,r))$, $\cd^b(\grmod_0\Gamma(p,q,r))$) admits a $\mathbb{Z}$-action via $\Sigma\circ\langle-1\rangle$. The Gabriel quiver of $\cd(\Gamma(p,q,r))$ (respectively, $\per(\Gamma(p,q,r))$, $\cd_{fd}(\Gamma(p,q,r))$) is the corresponding orbit quiver.  See Section~\ref{ss:orbit-category}. The category $\cd^b(\grmod_0 \Gamma(p,q,r))$ has Auslander--Reiten triangles, so does $\cd_{fd}(\Gamma(p,q,r))$. The Auslander--Reiten quiver of $\cd^b(\grmod_0\Gamma(p,q,r))$ admits a $\mathbb{Z}$-action via $\Sigma\circ\langle-1\rangle$, and the Auslander--Reiten quiver of $\cd_{fd}(\Gamma(p,q,r))$ is the corresponding orbit quiver.

\subsubsection{The case $p=0$}

By Sections~\ref{ss:representation-a-double-infty} and~\ref{ss:graded-module-over-graded-affine-A}, the
Gabriel quiver of $\cd^b(\Grmod\Gamma(0,q,r))$ consists of $4|r|\times\mathbb{Z}$ connected components: \begin{itemize}
\item[$\cdot$] $\cx_{j,i}$ of type $\mathbb{Z}A_\infty$,
\item[$\cdot$] $\cy_{j,i}$ of type $Q^l$,
\item[$\cdot$] $\cz_{j,i}$ of type $Q^l$,
\item[$\cdot$] $\ca_{j,i}$ of type $A_1$,
\end{itemize}
where $0\leq j\leq |r|-1$ and $i\in\mathbb{Z}$. For $C\in\{\cx,\cy,\cz,\ca\}$, the suspension functor $\Sigma$ acts as $\Sigma C_{j,i}=C_{j,i+1}$, and the degree shifting $\langle 1\rangle$ acts as $C_{j,i}\langle 1\rangle=C_{j-1,i}$.
The Gabriel quiver of $\ch^b(\grproj \Gamma(0,q,r))$ is formed by the components $\cx_{j,i}$ and $\cy_{j,i}$. The category $\cd^b(\grmod_0\Gamma(0,q,r))$ has Auslander--Reiten triangles, and its Auslander--Reiten quiver is formed by the components $\cx_{j,i}$. On each $\cx_{j,i}$, the Auslander--Reiten translation satisfies
$\tau^q=\langle -r\rangle$.

Therefore the Gabriel quiver of $\cd(\Gamma(0,q,r))\simeq\cd^b(\Grmod\Gamma(0,q,r))/\Sigma\circ\langle-1\rangle$ consists of
\begin{itemize}
\item[$\cdot$] $|r|$ components $\cx_j$ ($0\leq j\leq |r|-1$) of type $\mathbb{Z}A_\infty$,
\item[$\cdot$] $2|r|$ components $\cy_j$, $\cz_j$ ($0\leq j\leq |r|-1$) of type $Q^l$,
\item[$\cdot$] $|r|$ components $\ca_j$ ($0\leq j\leq |r|-1$) of type $A_1$.
\end{itemize}
The suspension functor acts as $\Sigma C_j=C_{j-1}$ ($0\leq j\leq |r|-1$) for $C\in\{\cx,\cy,\cz,\ca\}$.
The Gabriel quiver of $\per(\Gamma(0,q,r))\simeq\ch^b(\grproj\Gamma(0,q,r))/\Sigma\circ\langle-1\rangle$ is formed  by the components
$\cx_j$ and $\cy_j$, $0\leq j\leq |r|-1$. The category $\cd_{fd}(\Gamma(0,q,r))\simeq\cd^b(\grmod_0\Gamma(p,q,r))/\Sigma\circ\langle-1\rangle$ have Auslander--Reiten triangles and the Auslander--Reiten quiver of $\cd_{fd}(\Gamma(0,q,r))$ is formed by the components $\cx_j$, $0\leq j\leq |r|-1$. For each object $X$ in $\cx_j$, we have $\tau^q(X)=X\langle -r\rangle=\Sigma^{-r} X$.

\subsubsection{The case $p\neq 0$}\label{sss:AR-quiver-no-cycle}

By Sections~\ref{ss:representation-a-double-infty} and~\ref{ss:graded-module-over-graded-affine-A}, the
Gabriel quiver of $\cd^b(\Grmod\Gamma(p,q,r))$ consists of $8|r|\times\mathbb{Z}$ connected components:
\begin{itemize}
\item[$\cdot$] $\cp_{j,i}$ of type $\mathbb{Z}A_\infty^\infty$ (properly glued from $\mathbb{N}(Q^z)^{op}$ and $(-\mathbb{N})(Q^z)^{op}$),
\item[$\cdot$] $\cx^1_{j,i}$ of type $\mathbb{Z}A_\infty$,
\item[$\cdot$] $\cx^2_{j,i}$ of type $\mathbb{Z}A_\infty$,
\item[$\cdot$] $\cy^1_{j,i}$ of type $Q^l$,
\item[$\cdot$] $\cy^2_{j,i}$ of type $Q^l$,
\item[$\cdot$] $\cz^1_{j,i}$ of type $Q^l$,
\item[$\cdot$] $\cz^2_{j,i}$ of type $Q^l$,
\item[$\cdot$] $\ca_{j,i}$ of type $A_1$,
\end{itemize}
where $0\leq j\leq |r|-1$ and $i\in\mathbb{Z}$. For $C\in\{\cp,\cx^1,\cx^2,\cy^1,\cy^2,\cz^1,\cz^2,\ca\}$, the suspension functor $\Sigma$ acts as $\Sigma C_{j,i}=C_{j,i+1}$, and the degree shifting $\langle 1\rangle$ acts as $C_{j,i}\langle 1\rangle=C_{j-1,i}$.
The category $\ch^b(\grproj\Gamma(p,q,r))=\cd^b(\grmod_0\Gamma(p,q,r))$ has Auslander--Reiten triangles and its Auslander--Reiten quiver is formed by the components $\cp_{j,i}$, $\cx^1_{j,i}$ and $\cx^2_{j,i}$. On each $\cx^1_{j,i}$, the Auslander--Reiten translation acts as $\tau^q=\langle -r\rangle$, and on each $\cx^2_{j,i}$, the Auslander--Reiten translation acts as $\tau^p=\langle r\rangle$.

Therefore the Gabriel quiver of $\cd(\Gamma(p,q,r))\simeq\cd^b(\Grmod\Gamma(p,q,r))/\Sigma\circ\langle-1\rangle$ consists of
\begin{itemize}
\item[$\cdot$] $|r|$ components $\cp_j$ ($0\leq j\leq |r|-1$) of type $\mathbb{Z}A_\infty^\infty$,
\item[$\cdot$] $2|r|$ components $\cx^1_j$, $\cx^2_j$ ($0\leq j\leq |r|-1$) of type $\mathbb{Z}A_\infty$,
\item[$\cdot$] $4|r|$ components $\cy^1_j$, $\cy^2_j$, $\cz^1_j$, $\cz^2_j$ ($0\leq j\leq |r|-1$) of type $Q^l$,
\item[$\cdot$] $|r|$ components $\ca_j$ ($0\leq j\leq |r|-1$) of type $A_1$.
\end{itemize}
For $C\in\{\cp,\cx^1,\cx^2,\cy^1,\cy^2,\cz^1,\cz^2,\ca\}$, the suspension functor acts as $\Sigma C_j=C_{j-1}$ ($0\leq j\leq |r|-1$).
The category $\per(\Gamma(p,q,r))=\cd_{fd}(\Gamma(p,q,r))$ has Auslander--Reiten triangles.
Its Auslander--Reiten quiver is formed  by the components $\cp_j$, $\cx^1_j$ and $\cx^2_j$, $0\leq j\leq |r|-1$. For each object $X$ in $\cx^1_j$, we have $\tau^q(X)=X\langle -r\rangle=\Sigma^{-r} X$. For each object $X$ in $\cx^2_j$, we have $\tau^p(X)=X\langle r\rangle=\Sigma^{r}X$.

\subsection{The Koszul dual side}\label{ss:the-koszul-dual-side}

Let $(p,q,r)\in\Pi$. To each vertex $i$ of the graded quiver of $\Gamma=\Gamma(p,q,r)$, we associate a one-dimensional simple graded module $S_i$ which is concentrated in degree $0$. It is easily seen that \[\cd^b(\grmod_0\Gamma)=\thick_{\cd(\Grmod\Gamma)}(S_i\langle j\rangle| i=1,\ldots,p+q,j\in\mathbb{Z}).\] Therefore \[\cd_{fd}(\Gamma)=\cd^b(\grmod_0\Gamma)/\Sigma\circ\langle-1\rangle=\thick_{\cd(\Gamma)}(S_i|i=1,\ldots,p+q).\]
Moreover, each $S_i$ belongs to $\per(\Gamma)$, so $\cd_{fd}(\Gamma)\subseteq\per(\Gamma)$.

Let $\Gamma^*$ be the Koszul dual of $\Gamma$, that is, the dg endomorphism algebra of an $\ch$-projective resolution of $\bigoplus_{i=1}^{p+q}S_i$.

\subsubsection{The case $p=0$}  \label{sss:AR-quiver-truncated-cycle}
In this case, $\cd_{fd}(\Gamma)\subset\per(\Gamma)$. Therefore by~\cite[Lemma 10.5, the `symmetric' case]{Keller94}, we have an embedding $\cd(\Gamma^*)\rightarrow\cd(\Gamma)$. The essential image of $\cd(\Gamma^*)$ is $\Tria_{\cd(\Gamma)}(S_i|i=1,\ldots,q)$, the smallest triangulated subcategory of $\cd(\Gamma)$ containing the simples and closed under taking all direct sums. If we take the following $\ch$-projective resolutions of simple modules
\[
e_{i-1}\Gamma [-\delta_{i,1}r+1]\oplus e_i\Gamma ,~~d=\left(\begin{array}{cc}0 & 0\\ \alpha_{i-1} & 0\end{array}\right)
\]
(where the indices are taken modulo $q$),
then it is straightforward to check that $\Gamma^*$ has a dg subalgebra which has trivial differential and which is the quotient of the path algebra of the graded quiver
\[\xymatrix@R=0.7pc@C=0.7pc{&q\ar[r]&\cdot\ar@{.}[r]&\cdot\ar[r]&i+1\ar[dr]^{\alpha_i^*}&\\
1\ar[ur]^{\alpha_{q}^*}&&&&&i\ar[dl]^{\alpha_{i-1}^*},\\
&2\ar[ul]^{\alpha_1^*}&\cdot\ar[l]^{\alpha_2^*}\ar@{.}[r]&\cdot&i-1\ar[l]&}\]
modulo all paths of length two, where $\deg(\alpha_i^*)=1-\delta_{i,q}r$. It is clear that $\Gamma^*$ is graded equivalent to the quotient $\Gamma'(0,q,q-r)$ of $\Gamma(0,q,q-r)^{op}$ modulo all paths of length two. Notice that $\Gamma'(0,q,q-r)=\Gamma'(q,r)$, which is defined in Theorem~\ref{t:main-thm-1.1}.

If $r\neq 0$:  In this case the embedding $\cd(\Gamma^*)\to\cd(\Gamma)$ restricts to triangle equivalences $\per(\Gamma^*)\to\cd_{fd}(\Gamma)$ and $\cd_{fd}(\Gamma^*)\to\thick_{\cd(\Gamma)}(\Hom_k(\Gamma,k))$. So $\per(\Gamma^*)$ has Auslander--Reiten triangles and its Auslander--Reiten quiver consists of $|r|$ components $\cx^*_j$ ($1\leq j\leq |r|$) of type $\mathbb{Z}A_\infty$ (where $\tau=\Sigma^{-r}$ on objects). All indecomposable objects of $\Tria_{\cd(\Gamma)}(S_i|i=1,\ldots,q)$ belong to $\thick_{\cd(\Gamma)}(\Hom_k(\Gamma,k))$, whose  Gabriel quiver consists of the components $\cx_j$ and $\cz_j$ ($1\leq j\leq |r|$). So all indecomposable objects of $\cd(\Gamma^*)$ belong to $\cd_{fd}(\Gamma^*)$, whose Gabriel quiver of  consists of $\cx^*_j$ ($1\leq j\leq |r|$) and $|r|$ components $\cz^*_j$ ($1\leq j\leq |r|$) of type $Q^l$.

If $r=0$: The embedding $\cd(\Gamma^*)\to\cd(\Gamma)$ restricts to a triangle equivalence $\per(\Gamma^*)\to\cd_{fd}(\Gamma)$, which is exactly the bounded derived category of the standard tube of rank $q$. Therefore the Auslander--Reiten quiver of $\per(\Gamma^*)$ consists of $\mathbb{Z}$ connected components, each of which is a tube of rank $q$. The indecomposable objects of $\Tria_{\cd(\Gamma)}(S_i|i=1,\ldots,q)$ are shifts of indecomposable finite-dimensional $\Gamma$-modules and of the Pr\"ufer modules, the injective envelopes of the simple modules. As a consequence, the projective $\Gamma^*$-modules have to be sent to the Pr\"ufer modules. It follows that all indecomposable objects of $\cd(\Gamma^*)$ belong to $\cd_{fd}(\Gamma^*)$, whose Gabriel quiver consists of $\mathbb{Z}$ tubes of rank $q$ and $\mathbb{Z}$ cyclic quivers with $q$ vertices.

\subsubsection{The case $p\neq 0$}
In this case, we have $\per(\Gamma)=\cd_{fd}(\Gamma)$. Therefore by~\cite[Lemma 10.5, the `finite' case]{Keller94}, we have triangle equivalences $\cd(\Gamma)\simeq\cd(\Gamma^*)$ and $\per(\Gamma)\simeq\per(\Gamma^*)=\cd_{fd}(\Gamma^*)$.

\section{Graded gentle one-cycle algebras}\label{s:graded-gentle-one-cycle-alg}

In this section we show that a graded
gentle one-cycle algebra is derived equivalent to a graded quiver of type
$\tilde{A}_{p,q}$. In particular, with the results in Section~\ref{s:graded-a-pq}, this gives a new approach
to the description of the Auslander--Reiten quiver of the derived category of a gentle one-cycle algebra by
Bobi\'{n}ski--Geiss--Skowro\'nski~\cite{BobinskiGeissSkowronski04}.

\smallskip
Recall that there is an index set: $\Omega=\{(r,n,m)\in\mathbb{Z}^3|n\geq r\geq 1,m\geq 0\}$. We first introduce graded (respectively, trigraded) versions of $\Lambda(r,n,m)$ and bigraded versions of $\Gamma(p,q,r)$.
\begin{itemize}
\item[--]
For $(r,n,m)\in\Omega$ and $d\in\mathbb{Z}$ (respectively, $d\in\mathbb{Z}^3$), let $\Lambda(r,n,m,d)$ be the graded (respectively, trigraded) algebra whose underlying algebra is $\Lambda(r,n,m)$ and whose grading is defined by $\deg(\alpha_i)=\delta_{i,n-1}d$ ($-m\leq i\leq n-1$).
\item[--]
For integers $p,q,r_1,r_2$ with $(p,q,r_1),~(p,q,r_2)\in\Pi$, let $\Gamma(p,q,r_1,r_2)$ be the bigraded algebra whose underlying algebra is the complete path algebra of the quiver in the beginning of Section~\ref{s:graded-a-pq} with  $\mathrm{deg}(\alpha_i)=\delta_{i,p+q}(r_1,r_2)$. We have $F_1\Gamma(p,q,r_1,r_2)=\Gamma(p,q,r_2)$ and $F_2\Gamma(p,q,r_1,r_2)=\Gamma(p,q,r_1)$.
\end{itemize}

\begin{proposition}\label{p:A-L-derived-graded-here}
Let $d\in\mathbb{Z}$ and $(r,n,m)\in\Omega$.
\begin{itemize}
\item[(a)]
If $n>r$, the graded algebra $\Lambda(r,n,m,d)$ is derived equivalent to
$\Gamma(n-r,m+r,r-d)$.
\item[(b)]
If $n=r$, the graded algebra $\Lambda(n,n,m,d)$ is derived equivalent to
the quotient $\Gamma'(0,n+m,m+d)$ of $\Gamma(0,n+m,m+d)^{op}$ modulo
all paths of length $2$.
\end{itemize}
\end{proposition}

\begin{proof}
(a)
Let $p=n-r$ and $q=m+r$, and let $\Gamma$ be the bigraded algebra $\Gamma(p,q,r-d,r)$. Then $F_1\Gamma=\Gamma(p,q,r)$ and $F_2\Gamma=\Gamma(p,r,r-d)$. We will apply Proposition~\ref{p:derived-equiv-graded-alg-4} to this bigraded algebra. We
construct as follows a complex of bigraded $\Gamma$-modules such that forgetting the first grading we obtain a tilting object in $\cd(\Gamma(p,q,r))$ whose endomorphism
algebra is isomorphic to $\Lambda(r,n,m)$. Let $P_i$ denote the
indecomposable projective $\Gamma$-module corresponding to the
vertex $i$ and generated in bidegree $(0,0)$, and let $S_i$ be the simple top of $P_i$. Let
$T=\bigoplus_{i=-m}^{n-1}T_i$ be the complex of finitely generated projective bigraded $\Gamma$-modules given by (where $\underline{r}=(-r+d,-r)$)
\begin{align*}
T_i&=\begin{cases}
\Cone(P_{n}\stackrel{\alpha_{n+m+i}\cdots\alpha_n}{\longrightarrow}P_{n+m+i+1})\langle\underline{r}\rangle, & \text{ if } -m\leq i\leq -1,\\
\Cone(P_n\langle\underline{r}\rangle \stackrel{\alpha_{n+m}\cdots\alpha_n}{\longrightarrow}P_1) &\text{ if } i=0,\\
\Cone(P_{n-r+2-i}\oplus P_n\langle\underline{r}\rangle\stackrel{(\alpha_1\cdots\alpha_{n-r+1-i},\alpha_{n+m}\cdots\alpha_n)}{\longrightarrow}P_1), & \text{ if } 1\leq i\leq n-r,\\
\Sigma^{n-r-i}\Cone(P_{2n-r-i}\stackrel{\alpha_{2n-r-i}}{\longrightarrow}P_{2n-r+1-i})\langle\underline{r}\rangle, & \text{ if } n-r+1\leq i\leq n-1.\\
\end{cases}\end{align*}
For $n-r+1\leq i\leq n-1$, the $T_i$'s are $\Sigma S_n\langle\underline{r}\rangle$, $\Sigma^2S_{n+1}\langle\underline{r}\rangle$, \ldots, $\Sigma^{r-1}S_{n-r+2}\langle\underline{r}\rangle$.
For $-m\leq i\leq -1$ the composition series of the $T_i$'s are:
\noindent
\[
\begin{array}{c}
(n+1)\langle\underline{r}\rangle
\end{array},
\begin{array}{c}
(n+2)\langle\underline{r}\rangle\\
(n+1)\langle\underline{r}\rangle
\end{array},
\ldots,
\begin{array}{c}
(n+m)\langle\underline{r}\rangle\\
\vdots\\
(n+2)\langle\underline{r}\rangle\\
(n+1)\langle\underline{r}\rangle
\end{array}.
\]
For $0\leq i\leq n-r$ the composition series of the $T_i$'s are:
\[
\begin{array}{*{5}{c@{\hspace{-.5pt}}}}
&1&&&\\
2 & &&(n+m)\langle\underline{r}\rangle\\
\vdots&&&\vdots\\
n-r &&& (n+1)\langle\underline{r}\rangle\\
n-r+1&&&
\end{array},
\begin{array}{*{5}{c@{\hspace{-.5pt}}}}
&1&&&\\
2 & &&(n+m)\langle\underline{r}\rangle\\
\vdots&&&\vdots\\
n-r &&& (n+1)\langle\underline{r}\rangle\\
&&&
\end{array},\ldots,
\begin{array}{*{5}{c@{\hspace{-.5pt}}}}
&1&&&\\
2 & &&(n+m)\langle\underline{r}\rangle\\
&&&\vdots\\
&&& (n+1)\langle\underline{r}\rangle\\
&&&\\
\end{array},
\begin{array}{c}
1\\
(n+m)\langle\underline{r}\rangle\\
\vdots\\
(n+1)\langle\underline{r}\rangle\\
\mbox{}
\end{array}
\]

Direct computation shows that the trigraded endomorphism algebra of
$T$ in $\cd(\Grmod\Gamma)$
\[\bigoplus_{i,j,l}\Hom(T,\Sigma^i T\langle(j,l)\rangle)\]
is isomorphic to $\Lambda(r,n,m,r,-r+d,-r)$. It satisfies the condition required in Proposition~\ref{p:derived-equiv-graded-alg-4}, so $\Gamma(p,q,r-d)=F_2 \Gamma(p,q,r-d,r)$ is derived equivalent to $\Lambda(r,n,m,d)=\Tot F_2 \Lambda(r,n,m,r,-r+d,-r)$.

(b) Let $\Gamma'=\Gamma'(0,n+m,m+d,m)$. Let $P_i$ be the indecomposable
projective $\Gamma'$-module corresponding to the vertex $i$ and
generated in degree $0$. We construct a complex of bigraded $\Gamma$-modules such that forgetting the first grading we obtain a tilting object in $\cd(\Gamma')$ whose endomorphism algebra is
isomorphic to $\Lambda(n,n,m)$. For $-m\leq i\leq 0$,
let $T_i$ be the complex (the rightmost term is in degree $0$)
\[\xymatrix{P_{m+i+1}\ar[r]^{\alpha^{op}_{m+i}}&P_{m+i}\ar[r]^{\alpha^{op}_{m+i-1}}&\ldots\ar[r]&P_2\ar[r]^{\alpha^{op}_1}&P_1,}\]
which in composition series is
\[\xymatrix{{m+i+1\atop m+i+2}\ar[r]&{m+i\atop m+i+1}\ar[r]&\ldots\ar[r]&{2\atop 3}\ar[r]& {1\atop 2}.}\]
For $1\leq i\leq n-1$, let $T_i=P_{m+1+i}\langle d+m,m\rangle$. Put $T=\bigoplus_{i=-m}^{n-1}T_i$. Then the tigraded
dg endomorphism algebra of $T$ in $\cd(\Grmod\Gamma')$
\[\bigoplus_{i,j,l}\Hom(T,\Sigma^i T\langle(j,l)\rangle)\]
is isomorphic
to $\Lambda(n,n,m,-m,m+d,m)$. So $\Gamma'(0,n+m,m+d)=F_2 \Gamma'(0,n+m,m+d,m)$ is derived equivalent to $\Lambda(n,n,m,d)=\Tot F_2 \Lambda(n,n,m,-m,m+d,m)$.
\end{proof}

We give a simple example to illustrate Proposition~\ref{p:A-L-derived-graded-here} (a).

\begin{example}
By Proposition~\ref{p:A-L-derived-graded-here}, the hereditary algebra $\Gamma(1,1,0)$, the path algebra of the (ungraded) Kronecker quiver is derived equivalent to $\Lambda(1,2,0,1)$, the graded path algebra of the graded 2-cycle
\[\xymatrix{0\ar@<.7ex>[r]^{\alpha_0}&1\ar@<.7ex>[l]^{\alpha_1}}\]
with $\alpha_0$ in degree $0$ and $\alpha_1$ in degree $1$, modulo the graded ideal generated by the path $\alpha_0\alpha_1$. The graded hereditary algebra $\Gamma(1,1,1)$, the graded path algebra of the graded Kronecker quiver with one arrow in degree $0$ and the other arrow in degree $1$, is derived equivalent to $\Lambda(1,2,0)$, obtained from $\Lambda(1,2,0,1)$ above by forgetting the grading.
\end{example}

Recall from Lemma~\ref{lem:der-equiv-induces-equiv-on-per-and-dfd} that a derived equivalence  of dg algebras restricts to triangle equivalences on the perfect derived category $\per$ and the finite-dimensional derived category $\cd_{fd}$. Combining Proposition~\ref{p:A-L-derived-graded-here} and the results in Sections~\ref{ss:derived-cat-of-graded-tildeA} and~\ref{ss:the-koszul-dual-side}, we obtain

\begin{corollary}\label{cor:AR-quiver-of-graded-A-and-L}
Let $d\in\mathbb{Z}$ and $(r,n,m)\in\Omega$.
\begin{itemize}
\item[(a)] If $n>r$ and $d\neq r$, the Auslander--Reiten quiver of $\cd_{fd}(\Lambda(r,n,m,d))$ has exactly $3|r-d|$ components $\cp_j,~\cx_j^1,~\cx_j^2$ ($1\leq j\leq |r-d|$) of type $\mathbb{Z}A_\infty^\infty$, $\mathbb{Z}A_\infty$ and $\mathbb{Z}A_\infty$, respectively. For $X\in\cx_j^1$, we have $\tau^{m+r}X=\Sigma^{-r+d}X$ and for $X\in\cx_j^2$, we have $\tau^{n-r}=\Sigma^{r-d}X$.
\item[(b)] If $n=r$ and $d\neq n$, the Gabriel quiver of $\cd_{fd}(\Lambda(n,n,m,d))$ has exactly $2|n-d|$ components $\cx_j,~\cz_j$ ($1\leq j\leq |n-d|$) of type $\mathbb{Z}A_\infty$ and $Q^l$, respectively. $\per(\Lambda(n,n,m,d))$ has Auslander--Reiten triangles and its Auslander--Reiten quiver consists of $\cx_j$ ($1\leq j\leq |n-d|$). For $X\in\cx_j$, we have $\tau^{n+m}X=\Sigma^{-n+d}X$.
\end{itemize}
\end{corollary}

When $d=0$ we have $\Lambda(r,n,m,0)=\Lambda(r,n,m)$. In this case, Corollary~\ref{cor:AR-quiver-of-graded-A-and-L}  recovers \cite[Theorem B]{BobinskiGeissSkowronski04}.
In fact, we obtain a bit more. We also know the Gabriel quiver of $\cd(\Mod\Lambda(r,n,m))$ and that of $\ch(\Inj\Lambda(r,n,m))$, the homotopy category of complexes of injective $\Lambda(r,n,m)$-modules. For the former category, we have
\[\cd(\Mod\Lambda(r,n,m))\simeq\begin{cases} \cd(\Gamma(n-r,m+r,r)) & \text{ if } n\neq r,\\ \cd(\Gamma'(0,n+m,m)) & \text{ if } n=r.\end{cases}\]
For the latter category, we have $\ch(\Inj\Lambda(r,n,m))\simeq\cd(\Gamma(n-r,m+r,r))$. Indeed,
by~\cite[Appendix A]{Krause05} (see also~\cite[Remark of Corollary 4]{Yang12a}) we have $\ch(\Inj\Lambda(r,n,m))\simeq\cd(\Lambda(r,n,m)^*)$: if $n\neq r$, this is triangle equivalent to $\cd(\Gamma(n-r,m+r,r)^*)\simeq\cd(\Gamma(n-r,m+r,r))$; if $n=r$, this is triangle equivalent to $\cd(\Gamma'(0,n+m,m)^*)\simeq\cd(\Gamma(0,n+m,n))$ (see Section~\ref{ss:the-koszul-dual-side} for the last two $\simeq$). The category $\cd(\Mod \Lambda(r,n,m))$ is also described in \cite{ArnesenLakingPauksztelloPrest16}.

\smallskip
When we say a \emph{graded gentle one-cycle algebra}, we mean a gentle one-cycle algebra with a grading defined on each arrow. Observe that a graded $\Lambda(r,n,m)$ is graded equivalent to $\Lambda(r,n,m,d)$ for some $d\in\mathbb{Z}$.

\begin{theorem} \footnote{In Theorem~\ref{thm:graded-gentle-one-cycle-is-affine-A} the field $k$ was assumed to be algebraically closed. This assumption was needed only when we apply Theorem~\ref{t:gentle-one-cycle-clock} and Theorem\ref{t:discrete-derived-cat}(ii)$\Rightarrow$(iii). The referee pointed out that the proof of Theorem~\ref{t:gentle-one-cycle-clock} in \cite{AssemSkowronski87} and the proof of Theorem~\ref{t:discrete-derived-cat}(ii)$\Rightarrow$(iii) in \cite{BobinskiGeissSkowronski04} are field-independent and hence the assumption on $k$ is not necessary. We thank the referee for pointing this out to us.}\label{thm:graded-gentle-one-cycle-is-affine-A}
Let $A$ be a graded gentle one-cycle algebra.
\begin{itemize}
\item[(a)] If $A$ has finite global dimension, then it is derived equivalent to $\Gamma(p,q,r)$ for some $(p,q,r)\in\Pi$.
\item[(b)] If $A$ has infinite global dimension, then it is derived equivalent to $\Gamma'(0,q,r)$ for some $q\in\mathbb{N}$ and $r\in\mathbb{Z}$.
\end{itemize}
\end{theorem}
\begin{proof}
A tilting module $T$ over $FA$ is a direct sum of string modules, which are always gradable. Thus on $\End_{FA}(T)$ there is a natural graded algebra structure $B$, making $T$ a graded tilting $B$-$A$-module. So by Theorem~\ref{t:graded-rickard's-theorem}, $A$ and $B$ are graded derived equivalent.
As a consequence, it follows from Theorems~\ref{t:gentle-one-cycle-clock} and~\ref{t:discrete-derived-cat} that $A$ is graded derived equivalent to some graded algebra $B$, where $B$ is
\begin{itemize}
\item[$\cdot$] $\Gamma(p,q,r)$ for some $(p,q,r)\in\Pi$, if $FA$ satisfies the clock condition;
\item[$\cdot$] $\Lambda(r,n,m,d)$ for some $(r,n,m)\in\Omega$ with $n>r$ and some $d\in\mathbb{Z}$, if $FA$ does not satisfy the clock condition and has finite global dimension;
\item[$\cdot$] $\Lambda(n,n,m,d)$ for some $(n,n,m)\in\Omega$ and some $d\in\mathbb{Z}$, if $FA$ does not satisfy the clock condition and has infinite global dimension.
\end{itemize}
Thanks to Corollary~\ref{c:grade-derived-equiv-induce-derived-equiv}, $A$ is derived equivalent to $B$. Now applying Proposition~\ref{p:A-L-derived-graded-here}, we finish the proof.
\end{proof}

Let $A=kQ/I$ be a graded gentle one-cycle algebra. We define $d_+$
(respectively, $d_-$) as the difference between the number of
clockwise (respectively, counterclockwise) oriented relations and the
sum of the degrees of the clockwise (respectively, counterclockwise)
oriented arrows. We say that $A$ \emph{satisfies the graded clock
condition} if $d_+=d_-$.

\begin{conjecture} A graded gentle one-cycle algebra satisfies the graded clock condition  if and only if it is derived
equivalent to $\Gamma(p,q,0)$ for some $p,q>0$ or to $\Gamma'(0,q,q)$ (which is Koszul dual to $\Gamma(0,q,0)$) for some $q>0$.
\end{conjecture}

\def\cprime{$'$}
\providecommand{\bysame}{\leavevmode\hbox to3em{\hrulefill}\thinspace}
\providecommand{\MR}{\relax\ifhmode\unskip\space\fi MR }
\providecommand{\MRhref}[2]{%
  \href{http://www.ams.org/mathscinet-getitem?mr=#1}{#2}
}
\providecommand{\href}[2]{#2}

\end{document}